\documentclass[11pt,letterpaper]{article}
\usepackage{amsfonts, amsmath, amssymb, amscd, amsthm, graphicx,enumerate}
\usepackage{epsfig, color}
\usepackage{hyperref}

\makeatletter
\def\blfootnote{\xdef\@thefnmark{}\@footnotetext}
\makeatother

\newtheorem{thm}{Theorem}[section]
\newtheorem{cor}[thm]{Corollary}
\newtheorem{lem}[thm]{Lemma}
\newtheorem{prop}[thm]{Proposition}

\newtheorem{ques}[thm]{Question}
\newtheorem{conj}[thm]{Conjecture}
\theoremstyle{definition}
\newtheorem{defn}[thm]{Definition}
\theoremstyle{remark}
\newtheorem{rem}[thm]{Remark}
\newtheorem{ex}[thm]{Example}

\newtheorem{conv}[thm]{Convention}

\newfont{\eufm}{eufm10}

\newcommand{\ah}{{\mathcal{AH}^0_k}}
\newcommand{\ahc}{\overline{\mathcal{AH}}^0_k}
\newcommand{\G}{\Gamma (G, X\sqcup H)}

\newcommand{\Hl}{\{ H_\lambda \} _{\lambda \in \Lambda } }
\newcommand{\e}{\varepsilon }
\renewcommand{\phi}{\varphi}

\newcommand{\lab}{{\bf Lab}}

\renewcommand{\ll }{\langle\hspace{-.7mm}\langle }
\newcommand{\rr }{\rangle\hspace{-.7mm}\rangle }

\newcommand{\h}{\hookrightarrow _h }
\renewcommand{\d }{{\rm d}}
\newcommand{\dl }{\widehat{\rm d}}
\newcommand{\Lab }{{\bf Lab}}

\newcommand{\td}{{\rm td}}
\newcommand{\Aff}{\mbox{\it Aff}}

\begin{document}

\title{Transitivity degrees of countable groups and acylindrical hyperbolicity}
\author{M. Hull, D. Osin}

\date{}
\maketitle

\begin{abstract}
We prove that every countable acylindrically hyperbolic group admits a highly transitive action with finite kernel. This theorem uniformly generalizes many previously known results and allows us to answer a question of Garion and Glassner on the existence of highly transitive faithful actions of mapping class groups. It also implies that in various geometric and algebraic settings, the transitivity degree of an infinite group can only take two values, namely $1$ and $\infty$. Here by \emph{transitivity degree} of a group we mean the supremum of transitivity degrees of its faithful permutation representations. Further, for any countable group $G$ admitting a highly transitive faithful action, we prove the following dichotomy: Either $G$ contains a normal subgroup isomorphic to the infinite alternating group or $G$ resembles a free product from the model theoretic point of view. We apply this theorem to obtain new results about universal theory and mixed identities of acylindrically hyperbolic groups. Finally, we discuss some open problems.
\end{abstract}

\tableofcontents


\section{Introduction}

Recall that an action of a group $G$ on a set $\Omega$ is \emph{$k$-transitive} if $|\Omega |\ge k$ and for any two $k$-tuples of distinct elements of $\Omega$, $(a_1,..., a_k)$ and $(b_1,...,b_k)$, there exists $g\in G$ such that $ga_i=b_i$ for $i=1,...,k$. The \emph{transitivity degree} of a countable group $G$, denoted $\td(G)$, is the supremum of all $k\in \mathbb N$ such that $G$ admits a $k$-transitive faithful action. For finite groups, this notion is classical and fairly well understood. It is easy to see that $\td (S_n)=n$, $\td (A_n)=n-2$, and it is a consequence of the classification of finite simple groups that any finite group $G$ other than $S_n$ or $A_n$ has $\td (G)\le 5$. Moreover, if $G$ is not $S_n$, $A_n$, or one of the Mathieu groups $M_{11}$, $M_{12}$, $M_{23}$, $M_{24}$, then $td(G)\leq 3$ (see \cite{DixMor}).

For infinite groups, however, very little is known. For example, we do not know the answer to the following basic question: Does there exist an infinite countable group of transitivity degree $k$ for every $k\in \mathbb N$? There are examples for $k=1,2,3$, and $\infty$, but the problem seems open even for $k=4$. There is also a new phenomenon, which does not occur in the finite world: highly transitive actions. Recall that an action of a group is \emph{highly transitive} if it is $k$-transitive for all $k\in \mathbb N$. We say that a group is \emph{highly transitive} if it admits a highly transitive faithful action; it is easy to see that a countably infinite group is highly transitive if and only if it embeds as a dense subgroup in the infinite symmetric group $Sym (\mathbb N)$ endowed with the topology of pointwise convergence.  Obviously $\td (G)=\infty$ whenever $G$ is highly transitive, but we do not know if the converse is true. Yet another interesting question is whether there exists a reasonable classification of highly transitive groups (or, more generally, groups of high transitivity degree). The main goal of this paper is to address these questions in certain geometric and algebraic settings.

We begin by discussing known examples of highly transitive groups. Trivial examples of such groups are $FSym(\mathbb N)$ (the group of all finitary permutations of $\mathbb N$) and the infinite alternating group $Alt(\mathbb N)$ (the group of all finitary permutations that act as even permutations on their supports). Clearly every subgroup of $Sym (\mathbb N)$ that contains $Alt(\mathbb N)$ is also highly transitive. A fairly elementary argument allows one to construct finitely generated groups of this sort; we record the following.

\begin{prop}[Prop. \ref{ht-A}]\label{ext}
For every finitely generated infinite group $Q$, there exists a finitely generated group $G$ such that $FSym(\mathbb N)\lhd G \le Sym(\mathbb N)$ and $G/FSym(\mathbb N) \cong Q$. In particular, $G$ is highly transitive.
\end{prop}

This proposition can be used, for instance, to construct finitely generated highly transitive torsion groups by taking $Q$ to be finitely generated and torsion. Another interesting particular example of this sort is the family of Houghton groups $H_n$, $n\ge 2$, which correspond to the case $Q=\mathbb Z^{n-1}$. These groups are elementary amenable and finitely presented for $n\ge 3$ (see Example \ref{Hough}).

On the other hand, there are highly transitive groups of a completely different nature. The first such an example is $F_n$, the free group of rank $n\ge 2$. Highly transitive faithful actions of $F_n$ were constructed by McDonough \cite{McD}, and it was shown by Dixon \cite{Dix} that in some sense most finitely generated subgroups of $Sym(\mathbb N)$ are free and highly transitive. Other known examples of highly transitive groups include fundamental groups of closed surfaces of genus at least $2$ \cite{Kit} and, more generally, non-elementary hyperbolic groups without non-trivial finite normal subgroups \cite{Cha}, all free products of non-trivial groups except $\mathbb Z_2\ast \mathbb Z_2$ \cite{GM, Gun, Hic, MS} and many other groups acting on trees \cite{FMS}. Garion and Glassner \cite{GG} used an interesting approach employing Tarski Monster groups constructed by Olshanskii \cite{Ols80} to show that $Out(F_n)$ is highly transitive for $n\ge 4$. Motivated by similarity between $Out(F_n)$ and mapping class groups, they ask which mapping class groups are highly transitive. Another question from \cite{GG} is whether $Out(F_3)$ is highly transitive.

The first main result of our paper is a theorem which simultaneously generalizes most results mentioned in the previous paragraph and allows us to answer the questions of Garion and Glassner. Recall that an isometric action of a group $G$ on a metric space $(X, d)$ is \emph{acylindrical} if for all $\e>0$, there exist $R>0, N>0$ such that for all $x, y\in X$ with $d(x, y)\geq R$, there are at most $N$ elements $g\in G$ satisfying $d(x, gx)\leq \e$ and $d(y, gy)\leq \e$. For example, it is easy to show that the action of a fundamental group of a graph of groups with finite edge groups on the associated Bass-Serre tree is acylindrical. A group is called \emph{acylindrically hyperbolic} if it is not virtually cyclic and admits an acylindrical action on a hyperbolic metric space with unbounded orbits \cite{Osi13}. The class of acylindrically hyperbolic groups includes many examples of interest; for an extensive list we refer to \cite{DGO,MO,Osi13,Osi14}.

Every acylindrically hyperbolic group $G$ has a maximal normal finite subgroup called the \emph{finite radical} of $G$ and denoted $K(G)$ \cite{DGO}.

\begin{thm}\label{main}
Every countable acylindrically hyperbolic group admits a highly transitive action with finite kernel. In particular, every countable acylindrically hyperbolic group with trivial finite radical is highly transitive.
\end{thm}

It is well-known and easy to see that a group of transitivity degree at least $2$ cannot have finite normal subgroups (see Lemma \ref{2-trans} (b)).  Thus the kernel of a highly transitive action of a countable acylindrically hyperbolic group $G$ cannot be smaller than $K(G)$. In particular, the assumption about triviality of finite radical cannot be dropped from the second sentence of the theorem.

Our proof of Theorem \ref{main} is based on the notion of a small subgroup of an acylindrically hyperbolic group introduced in Section 2. This notion seems to be of independent interest. In addition, the proof of Theorem \ref{main} makes use of hyperbolically embedded subgroups introduced in \cite{DGO} and small cancellation theory in acylindrically hyperbolic groups developed in \cite{H13}.

Here we mention just few particular cases of our theorem considered in Corollaries \ref{mcg0}, \ref{3d0}, and \ref{RAAG}--\ref{1-rel}. By $\Sigma_{g,n,p}$  we denote a $p$ times punctured compact orientable surface of genus $g$ with $n$ boundary components.

\begin{cor}\label{maincor-1}
\begin{enumerate}
\item[(a)] Let $G$ be a countable group hyperbolic relative to a collection of proper subgroups. Then $G$ is highly transitive if and only if it is not virtually cyclic and has no non-trivial finite normal subgroups.
\item[(b)] For $g,n,p\in \mathbb N\cup\{0\}$, the mapping class group ${\rm Mod}(\Sigma _{g,n,p})$  is highly transitive if and only if $n=0$, $3g+p\ge 5$, and $(g,p)\ne (2,0)$.
\item[(c)] (cf. \cite[Theorem 1]{GG}) $Out(F_n)$ is highly transitive if and only if $n\ge 3$.
\item[(d)] Let $M$ be a compact irreducible $3$-manifold. Then $\pi_1(M)$  is highly transitive if and only if it is not virtually solvable and $M$ is not Seifert fibered.
\item[(e)] A right angled Artin group is highly transitive if and only if it is non-cyclic and directly indecomposable.
\item[(f)] Every $1$-relator group with at least $3$ generators is highly transitive.
\end{enumerate}
\end{cor}

Recall that a free product of two non-trivial groups has no non-trivial finite normal subgroups. Thus part (a) of Corollary \ref{maincor-1} covers the case of free products considered in \cite{GM, Gun, Hic} as well as some results of \cite{FMS}; it also covers the result about hyperbolic groups \cite{Cha}. Parts (b) and (c) answer the questions of Garion and Glassner mentioned above. Our contribution to (c) is the case $n=3$.

Theorem \ref{main} also allows us to show that in various geometric and algebraic settings, transitivity degree of an infinite group can only take two values, namely $1$ and $\infty$, and infinite transitivity degree is equivalent to being highly transitive. More precisely, we consider three conditions for a group $G$:

\begin{enumerate}
\item[(C$_1$)] $\td (G)=\infty$.
\item[(C$_2$)] $G$ is highly transitive.
\item[(C$_3$)] $G$ is acylindrically hyperbolic with trivial finite radical.
\end{enumerate}

In Section 4, we show that an infinite subgroup $G$ of a mapping class group is either acylindrically hyperbolic or satisfies certain algebraic conditions which imply  $\td (G)=1$. A similar result for subgroups of $3$-manifold groups was obtained in \cite{MO}. Combining this with Theorem \ref{main}, we obtain the following.

\begin{cor}[Cor. \ref{mcg1}, Cor. \ref{3d1}]
Let $G$ be an infinite subgroup of ${\rm Mod}(\Sigma _{g,n,p})$ for some $g,n,p \in \mathbb N\cup\{0\}$ or an infinite subgroup of the fundamental group of a compact $3$-manifold. Then $\td(G)\in \{ 1, \infty\}$ and conditions (C$_1$)--(C$_3$) are equivalent.
\end{cor}

A similar result can also be proved in some algebraic settings, e.g., for subgroups of finite graph products. Let $\Gamma $ be a graph (without loops or multiple edges) with vertex set $V$ and let  $\{G_v\}_{v\in V}$  be a family of groups indexed by vertices of $\Gamma$. The \emph{graph product} of $\{G_v\}_{v\in V}$ with respect to $\Gamma $, denoted  $\Gamma \{G_v\}_{v\in V}$, is the quotient group of the free product $\ast_{v\in V} G_v$ by the relations $[g,h]=1$ for all $g\in G_u$, $h\in G_v$ whenever $u$ and $v$ are adjacent in $\Gamma$. Graph products simultaneously generalize free and direct products of groups. Basic examples are right angled Artin and Coxeter groups, which are graph products of copies of $\mathbb Z$ and $\mathbb Z_2$, respectively. The study of graph products and their subgroups has gained additional importance in view of the recent breakthrough results of Agol, Haglund, Wise, and their co-authors, which show that many groups can be virtually embedded into right angled Artin groups (see \cite{Agol,HW,Wise-qch} and references therein).

\begin{cor}[Cor. \ref{GP}]\label{maincor2}
Let $G$ be a countably infinite subgroup of a finite graph product $\Gamma \{G_v\}_{v\in V}$. Suppose that $G$ is not isomorphic to a subgroup of one of the multiples. Then $\td (G)\in \{ 1,2,\infty\}$  and conditions (C$_1$)--(C$_3$) are equivalent. If, in addition, every $G_v$ is residually finite, then $\td(G)\in \{ 1, \infty\}$.
\end{cor}

To state our next result we need some preparation. Let $F_n$ denote the free group of rank $n$ and recall that a group $G$ satisfies a \emph{mixed identity} $w=1$ for some $w\in G\ast F_n$ if every homomorphism $G\ast F_n\to G$ that is identical on $G$ sends $w$ to $1$. We say that the mixed identity $w=1$ is non-trivial if $w\ne 1$ as an element of $G\ast F_n$. For the general theory of mixed identities and mixed varieties of groups we refer to \cite{Ana}. We say that $G$ is \emph{mixed identity free} (or \emph{MIF} for brevity) if it does not satisfy any non-trivial mixed identity.

The property of being MIF is much stronger than being identity free and imposes strong restrictions on the algebraic structure of $G$. For example, if $G$ has a non-trivial center, then it satisfies the non-trivial mixed identity $[a,x]=1$, where  $a\in Z(G)\setminus\{ 1\}$. Similarly, it is easy to show that an MIF group has no finite normal subgroups, is directly indecomposable, has infinite girth, etc. (see Proposition \ref{mif-prop}). Other examples of groups satisfying a non-trivial mixed identity are Thompson's group $F$ \cite{Zar} or any subgroup of $Sym(\mathbb N) $ containing $Alt(\mathbb N)$ (see Theorem \ref{dichotomy} below).

It is also worth noting that MIF groups resemble free products from the model theoretic point of view. More precisely, a countable  group $G$ is MIF if and only if $G$ and $G\ast F_n$ are universally equivalent as $G$-groups for all $n\in \mathbb N$. This means that a universal first order sentence with constants from $G$ holds true in $G$ if and only if it holds true in $G\ast F_n$. Groups universally equivalent to $G$ as $G$-groups are exactly the coordinate groups of irreducible algebraic varieties over $G$ \cite{BMR}. For more details we refer to Section 5 and Proposition \ref{mif}.

\begin{figure}
\centering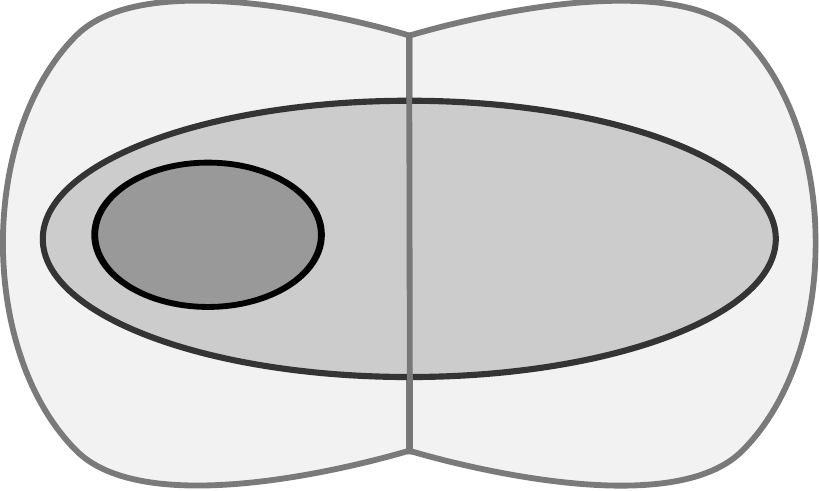\\
  \caption{}\label{classes}
\end{figure}

In Section 5, we prove the following dichotomy for highly transitive countable groups, which captures the principal difference between examples provided by Proposition \ref{ext} and Theorem \ref{main}.

\begin{thm}\label{dichotomy}
Let $G$ be a highly transitive countable group. Then exactly one of the following two mutually exclusive conditions holds.
\begin{enumerate}
\item[(a)] $G$ contains a normal subgroup isomorphic to $Alt(\mathbb N)$.
\item[(b)] $G$ is MIF (equivalently, $G$ and $G\ast F_n$ are universally equivalent as $G$-groups for every $n\in \mathbb N$).
\end{enumerate}
\end{thm}

It is known that infinite normal subgroups of acylindrically hyperbolic groups are also acylindrically hyperbolic and hence cannot be torsion \cite{Osi13}. Thus condition (a) from Theorem \ref{dichotomy} cannot hold if $G$ is acylindrically hyperbolic. Combining this with Theorem \ref{main} we obtain the following.

\begin{cor}[Cor. \ref{ah-mif}]
Let $G$ be an acylindrically hyperbolic group with trivial finite radical. Then $G$ is MIF.
\end{cor}

Clearly $G$ is not MIF whenever $K(G)\ne \{ 1\}$. In the particular case of non-cyclic torsion free hyperbolic groups, the above corollary was proved in \cite{AR} by different methods.

Theorem \ref{main} and Theorem \ref{dichotomy} are summarized in Fig. \ref{classes}. Here by $\mathcal{AH}^0$ we denote the class of acylindrically hyperbolic groups with trivial finite radical. In Section 5 we show that all inclusions are proper. We mention two examples.
\begin{prop}[Cor. \ref{BS12}, Cor. \ref{mif-tor}]\label{prop-ex}
\begin{enumerate}
\item[(a)] The group $G=\langle a,b,t\mid a^b=a^2, [a,t]=1\rangle$ is highly transitive and MIF, but not acylindrically hyperbolic.
\item[(b)] There exist finitely generated MIF groups of transitivity degree $1$.
\end{enumerate}
\end{prop}

The paper is organized as follows: In Section 2 we review properties of acylindrically hyperbolic groups and hyperbolically embedded subgroups necessary for the proof of Theorem \ref{main}. In Section 3 we prove Theorem \ref{main}, and in Section 4 we apply this theorem to various classes of groups to obtain Corollaries \ref{maincor-1}-\ref{maincor2}. In Section 5 we study the relationship between mixed identities and highly transitive actions and prove Theorem \ref{dichotomy}. Some open questions and relevant examples are discussed in Section 6.

\paragraph{Acknowledgments.} We would like to thank Ilya Kapovich, Olga Kharlampovich, and Alexander Olshanskii for helpful discussions of various topics related to this paper. We are especially grateful to Dan Margalit for answering our numerous questions on mapping class groups. The second author was supported by the NSF grant DMS-1308961.


\section{Preliminaries on acylindrically hyperbolic groups}\label{sec2}


\paragraph{2.1. Generating alphabets and Cayley graphs.}
When dealing with relative presentations of groups, we often need to represent the same element of a group $G$ by several distinct generators. Thus, instead of a generating set of $G$, it is more convenient to work with an alphabet $\mathcal A$ given together with a (not necessarily injective) map $\alpha\colon \mathcal A\to G$ such that $\alpha(\mathcal A)$ generates $G$. We begin by formalizing this approach.

Let $\mathcal A$ be a set, which we refer to as an \emph{alphabet}. Let $G$ be a group and let $\alpha\colon \mathcal A\to G$ be a (not necessarily injective) map. We say that $G$ is \emph{generated} by $\mathcal A$ (or $\mathcal A$ is a \emph{generating alphabet} of $G$) if $G$ is generated by $\alpha(\mathcal A)$. Note that a generating set $X\subseteq G$ can be considered as a generating alphabet with the obvious map $\alpha\colon X\to G$.

By the \emph{Cayley graph} of $G$ with respect to a generating alphabet $\mathcal A$, denoted $\Gamma (G, \mathcal A)$, we mean a graph with vertex set $G$ and the set of edges defined as follows. For every $a\in \mathcal A$ and every $g\in G$, there is an oriented edge $(g, g\alpha(a))$ in $\Gamma (G, \mathcal A)$ labelled by $a$. Given a (combinatorial) path $p$ in $\Gamma (G, \mathcal A)$, we denote by $\Lab (p)$ its label. Note that if $\alpha $ is not injective, $\Gamma (G, \mathcal A)$ may have multiple edges. Of course, the identity map on $G$ induces an isometry between vertex sets of the graphs $\Gamma (G, \mathcal A)$ and $\Gamma (G, \alpha(\mathcal A))$. Thus we only need to distinguish between $\Gamma (G, \mathcal A)$ and $\Gamma (G, \alpha(\mathcal A))$ when dealing with labels; in all purely metric considerations we do not need generating alphabets and can simply work with generating sets.

Given a generating set $X$ of $G$ and an element $g\in G$, let $|g|_{X}$ denote the word length of $g$ with respect to $X$, that is the length of a shortest word in  $X^{\pm 1}$ that represents $g$. For $g, h\in G$, we define $d_{X}(g, h)=|g^{-1}h|_{X}$. Finally, let $$B_{X}(n)=\{g\in G\;|\; |g|_{X}\leq n\}.$$

\paragraph{2.2. Hyperbolically embedded subgroups.}
The typical situation when we apply the language described above is the following. Suppose that we have a group $G$, a subgroup $H$ of $G$, and a subset $X\subseteq G$ such that $X$ and $H$ together generate $G$. We think of $X$ and $H$ as abstract sets and consider the disjoint union
\begin{equation}\label{calA}
\mathcal A= X \sqcup H,
\end{equation}
and the map $\alpha \colon \mathcal A\to G$  induced by the obvious maps $X\to G$ and $H\to G$. By abuse of notation, we do not distinguish bretween subsets $X$ and $H$ of $G$ and their preimages in $\mathcal A$. This will not create any problems since the restrictions of $\alpha$ on $X$ and $H$ are injective. Note, however, that $\alpha $ is not necessarily injective. Indeed if $X$ and subgroup $H$ intersect in $G$, then every element of $H \cap X\subseteq G$ will have at least two preimages in $\mathcal A$: one in $X$ and another in $H$ (since we use disjoint union in (\ref{calA})).

\begin{conv}
Henceforth we always assume that generating sets and relative generating sets are symmetric. That is, if $x\in X$, then $x^{-1}\in X$. In particular, every element of $G$ can be represented by a word in  $X\sqcup H$.
\end{conv}

In these settings, we consider the Cayley graphs $\G $ and $\Gamma (H, H)$ and naturally think of the latter as a subgraph of the former. We introduce a \textit{relative metric} $\dl \colon H \times H \to [0, +\infty]$ as follows. We say that a path $p$ in $\G$ is {\it admissible} if it contains no edges of $\Gamma (H, H)$. Let $\dl (h,k)$ be the length of a shortest admissible path in $\G $ that connects $h$ to $k$. If no such a path exists, we set $\dl (h,k)=\infty $. Clearly $\dl $ satisfies the triangle inequality, where addition is extended to $[0, +\infty]$ in the natural way.

\begin{defn}\label{hedefn}
A subgroup $H$ of $G$ is \emph{hyperbolically embedded  in $G$ with respect to a subset $X\subseteq G$}, denoted $H \h (G,X)$, if the following conditions hold.
\begin{enumerate}
\item[(a)] The group $G$ is generated by $X$ together with $H$ and the Cayley graph $\G $ is hyperbolic.
\item[(b)] Any ball (of finite radius) in $H$ with respect to the metric $\dl$ contains finitely many elements.
\end{enumerate}
Further we say  that $H$ is \emph{hyperbolically embedded} in $G$ and write $H\h G$ if $H\h (G,X)$ for some $X\subseteq G$.
\end{defn}

Hyperbolically embedded subgroups were introduced and studied in \cite{DGO} as a generalizaton of relatively hyperbolic groups; indeed, $G$ is hyperbolic relative to $H$ if and only if $H\h(G, X)$ with $|X|<\infty$ \cite[Proposition 4.28]{DGO}.
The following lemma is a particular case of this result.

\begin{lem}\label{heprod}
For any group $G$, $G\h(G\ast\langle t\rangle, \{t\})$.
\end{lem}

In the next two lemmas, we let $G$ be a group and $H$ a subgroup of $G$.

The following is a particular case of \cite[Proposition 4.35]{DGO}.
\begin{lem}\label{hehe}
Suppose that $H\h(G, X)$ for some $X\subseteq G$ and there is a subset $Y\subseteq H$ and a subgroup $K \h (H, Y)$. Then $K\h (G, X\cup Y)$.
\end{lem}

Given a group $G$, a subgroup $H\le G$, and $g\in G$, we denote by $H^g$ the conjugate $g^{-1} Hg$.

\begin{lem}\label{malnorm}\cite[Proposition 2.8]{DGO}
If $H\h G$, then for all $g\in G\setminus H$, $|H\cap H^g|<\infty$.
\end{lem}

\paragraph{2.3. Acylindrically hyperbolic groups.}
 Recall that an isometric action of a group $G$ on a metric space $S$ is {\it acylindrical} if for every $\e>0$, there exist $R,N>0$
such that for every two points $x,y$ with $\d (x,y)\ge R$, there are at most $N$ elements $g\in G$ satisfying
$$
\d(x,gx)\le \e \;\;\; {\rm and}\;\;\; \d(y,gy) \le \e.
$$

Given a group $G$ acting on a hyperbolic space $S$, an element $g\in G$ is called \emph{loxodromic} if the map $\mathbb Z\to S$ defined by $n\mapsto g^ns$ is a quasi-isometry for some (equivalently, any) $s\in S$. Every loxodromic element $g\in G$ has exactly $2$ limit points $g^{\pm\infty}$ on the Gromov boundary $\partial S$. Loxodromic elements $g,h\in G$ are called \emph{independent} if the sets $\{ g^{\pm \infty}\}$ and $\{ h^{\pm \infty}\}$ are disjoint.

We will often use the following.

\begin{thm}[{\cite[Theorem 1.1]{Osi13}}]\label{tri}
Let $G$ be a group acting acylindrically on a hyperbolic space. Then exactly one of the following three conditions holds.
\begin{enumerate}
\item[(a)] $G$ is elliptic.
\item[(b)] $G$  contains a loxodromic element $g$ such that $|G:\langle g\rangle |<\infty$.
\item[(c)] $G$ contains infinitely many loxodromic elements that are independent (i.e., have disjoint limit sets on the boundary).
\end{enumerate}
\end{thm}

\begin{defn}\label{df:acyl-gp}
A group $G$ is called \emph{acylindrically hyperbolic} if it admits a non-elementary acylindrical action on a hyperbolic space.
\end{defn}

In the case of acylindrical actions on hyperbolic spaces being non-elementary is equivalent to the action having unbounded orbits and $G$ being not virtually cyclic by theorem \ref{tri}.

It is easy to see from the definition that for every group $G$ we have $G\h G$ and $K\h G$ for every finite subgroup $K$ of $G$. Following \cite{DGO}, we call a hyperbolically embedded subgroup $H\h G$ \emph{non-degenerate} if $H$ is infinite and $H\ne G$.

The next result is a part of \cite[Theorem 2.2]{Osi13}. In particular, it allows us to apply all results from \cite{DGO} concerning groups with non-degenerate hyperbolically embedded subgroups to acylindrically hyperbolic groups.

\begin{thm}\label{ah-hes}
A group $G$ is acylindrically hyperbolic if and only if $G$ contains a non-degenerate hyperbolically embedded subgroup.
\end{thm}

If $G$ is acylindrically hyperbolic, then $G$ has a unique, maximal finite normal subgroup $K(G)$ called the \emph{finite radical} of $G$ \cite[Theorem 6.14]{DGO}.

\begin{lem}\cite{H13}\label{finrad}
Let $G$ be acylindrically hyperbolic. Then $G/K(G)$ is acylindrically hyperbolic and has trivial finite radical.
\end{lem}

\paragraph{2.4. Small subgroups in acylindrically hyperbolic groups.}  Given a group $G$ and a generating set (or an alphabet) $\mathcal A$, we say that the Cayley graph $\Gamma (G,\mathcal A)$ is \emph{acylindrical} if so is the action of $G$ on $\Gamma (G,\mathcal A)$. A hyperbolic space $S$ is called \emph{non-elementary} if $|\partial S|>2$, where $\partial S$ denotes the Gromov boundary of $S$.

\begin{defn}
We call a subgroup $H$ of a group $G$ \emph{small in $G$} (or just \emph{small} if $G$ is understood) if there exists a generating set $X$ of $G$ such that $H\subseteq X$ and $\Gamma (G,X)$ is hyperbolic, non-elementary, and acylindrical.
\end{defn}

Note that every group containing small subgroups is acylindrically hyperbolic.

\begin{lem}\label{faith}
Assume that a group $G$ is acylindrically hyperbolic and $K(G)=\{ 1\}$. Then for every small subgroup $H\le G$, the action of $G$ on the coset space $G/H$ is faithful.
\end{lem}

\begin{proof}
Let $N$ denote the kernel of the action of $G$ on $G/H$. Let $X$ be a generating set of $G$ such that $H\subseteq X$ and $\Gamma (G,X)$ is acylindrical, hyperbolic, and non-elementary. Note that $N$ acts elliptically on $\Gamma (G,X)$ as $N\le H\subseteq X$. Recall that for every acylindrical non-elementary action of a group on a hyperbolic space, every normal elliptic subgroup is finite \cite[Lemma 7.2]{Osi13}. Applying this to the group $N$ acting on $\Gamma (G,X)$, we obtain that $N\le K(G)$ and consequently $N=\{ 1\}$.
\end{proof}

The proof of Theorem \ref{main} will make use of one particular example of small subgroups. Recall that given a generating set $X$ of a group $G$, $|\cdot |_X$ and $\d_X$ denote the corresponding word length and metric on $G$, respectively.

\begin{prop}\label{small}
Let $G$ be a group, $F$ a subgroup of $G$, $X$ a subset of $G$ such that $F\h (G,X)$. Suppose that a subgroup $H\le G$ is generated by a subset $Z$ such that $$\sup_{z\in Z} |z|_X<\infty$$ and $$\d_Z (h_1, h_2)\le K \d_{X\cup F} (h_1, h_2)$$ for all $h_1, h_2\in H$. Then $H$ is small. Moreover, for every subgroup $S\le G$ that properly contains $F$, there exists a subset $Y\subseteq G$ such that $\Gamma (G, Y)$ is hyperbolic and acylindrical, $H\subseteq Y$,  and the action of $S$ on $\Gamma (G, Y)$ is non-elementary.
\end{prop}

The proof can be easily extracted from \cite{Osi13}. In order to make our paper self-contained, we recall necessary definitions and results from \cite{Osi13}.

Let $\mathcal A=X\sqcup F$ and let $p$ be a (combinatorial) path in $\Gamma (G, \mathcal A)$. A non-trivial subpath $q$ of $p$ is called an \emph{$F$-component} if $\lab (q)$ is a word in the alphabet $F$ and $q$ is not contained in any bigger subpath of $p$ with this property.  If $q$ is an $F$-component of $p$ and $a,b$ are endpoints of $q$, then all vertices of $q$ belong to the same coset $aF=bF$; in this case we say that $p$ \emph{penetrates} $aF$ and call the number $\dl (1, a^{-1}b)$ the \emph{depth of the penetration}. Given two elements $f,g\in G$, we say that a coset $aF$ is \emph{$(f,g;D)$-separating} if there exists a geodesic $p$ from $f$ to $g$ in $\Gamma (G,\mathcal A)$ such that $p$ penetrates $aF$ and the depth of the penetration is greater than $D$.

The term ``separating" is justified by the following result, which is a simplification of \cite[Lemma 4.5]{Osi13}.

\begin{lem}\label{sep}
Let $G$ be a group, $F$ a subgroup of $G$, $X$ a subset of $G$ such that $F\h (G,X)$. Then for any constants $\mu $ and $b$, there exists $C=C(\mu, b)$ such that the following holds. Let $f,g\in G$ and let $aF$ be an $(f,g;C)$-separating coset. Then every $(\mu, b)$-quasi-geodesic in $\Gamma (G, X\sqcup F)$ that connects $f$ to $g$ penetrates $aF$.
\end{lem}

It is observed in \cite{Osi13} that if $F\h (G,X)$, then the Cayley graph $\Gamma (G,X\sqcup H)$ is not necessarily acylindrical.  However, the following lemma holds true. Its first claim is a particular case of Theorem 5.4 in \cite{Osi13}; the second claim is not stated in that theorem itself, but is obvious from its proof.

\begin{lem}\label{Y}
Let $G$ be a group, $F$ a subgroup of $G$, $X$ a subset of $G$ such that $F\h (G,X)$. Then there exists a set $Y\supseteq X$ such that $F\h (G,Y)$ and $\Gamma (G,Y\sqcup F)$ is acylindrical. More precisely, given a constant $D>0$, let $Y=Y(D)$ denote the set of all elements $y\in G$ such that the set of all $(1,y;D)$-separating cosets is empty. Then for every large enough $D$, the set $Y$ satisfies the above requirements.
\end{lem}

We are now ready to prove Proposition \ref{small}.

\begin{proof}
Let $M=\sup_{z\in Z} |z|_X$. For every element $h\in H$, let $p_h$ be a shortest path in $\Gamma (G, X\sqcap H)$ connecting $1$ to $h$ such that $\lab (p_h)$ is a word in the alphabet $X$. In particular, $p_h$ does not penetrate any left coset of $F$. It is straightforward to check that $p_h$ is $(\mu, b)$-quasi-geodesic for some $\mu$ and $b$ that only depend on $K$ and $M$ (in fact, we can take $\mu=KM$ and $b=2KM^2$).

Let $D$ be a constant such that the conclusion of Lemma \ref{Y} holds and $D\ge C(\mu,b)$, where $C(\mu,b)$ is given by Lemma \ref{sep}. Then for any $h\in H$, there are no $(1,h; D)$-separating cosets. Therefore, the set $Y=Y_D$ defined in Lemma \ref{Y} contains $H$. By Lemma \ref{Y}, $\Gamma (G,Y\sqcup F)$ is acylindrical.

It remains to show that the action of $S$ on $\Gamma (G, Y\sqcup F)$ is non-elementary. By Theorem \ref{tri}, it suffices to show that $S$ has unbounded orbits and is not virtually cyclic. Let $a\in S\setminus F$. By \cite[Corollary 6.12]{DGO}, there exists $f\in F$ such that $af$ is loxodromic with respect to the action on $\Gamma (G, Y\sqcup F)$. Since $af\in S$, we obtain that the orbits of $S$ are unbounded. Finally assume that $S$ is elementary. Since $F$ is infinite, it must be of finite index in $S$. In particular, $F$ contains an infinite subgroup $N$ such that $N\lhd S$. Then for every $a\in S$, we obtain $|F^a\cap F|\ge |N|=\infty$ which contradicts Lemma \ref{malnorm} if $a\notin F$. Thus the action of $S$ on $\Gamma (G, Y\sqcup F)$ is non-elementary.
\end{proof}


\section{Highly transitive actions of acylindrically hyperbolic groups}


\paragraph{3.1. Outline of the proof.}
The goal of this section is to prove Theorem \ref{main}. We begin with a brief sketch of the proof.

We first note that by Lemma \ref{finrad}, it suffices to deal with the case of trivial finite radical. Assuming $K(G)=\{ 1\}$, we will construct a subgroup $H_\infty\leq G$ such that the action of $G$ on the left cosets of $H_\infty$ will be highly transitive. $H_\infty$ will be defined as the union of a sequence of subgroups $$H_1\leq H_2\leq...$$ Given $H_i$ and two tuples of pairwise distinct cosets $(a_1H_i,..., a_kH_i)$ and $(b_1H_i,..., b_kH_i)$, we define $H_{i+1}$ by choosing an element $g_i\in G$ and setting
\begin{equation}\label{Hi+1}
H_{i+1}=\langle H_i, b_1^{-1}g_ia_1,..., b_k^{-1}g_ia_k\rangle.
\end{equation}
Then clearly $g_ia_jH_{i+1}=b_jH_{i+1}$, so after enumerating all tuples of elements of $G$ and repeating this process inductively we will get that the action of $G$ on the set $\Omega $ of left cosets of $H_\infty$ is highly transitive, provided $|\Omega|=\infty$. To the best of our knowledge, this natural construction first appeared in \cite{BO}. It was also used by Chaynikov \cite{Cha} to prove Theorem \ref{main} in the particular case when $G$ is hyperbolic.

The main difficulty is to ensure that that the action of $G$ on $\Omega$ is faithful. Note that if the elements $g_i$ is not chosen carefully, then we are likely get $H_\infty=G$. In Chainikov's paper \cite{Cha}, the proof of faithfulness of the action is built on the well-developed theory of quasi-convex subgroups of hyperbolic groups. Currently there is no analogue of this theory in the more general context of acylindrically hyperbolic groups, so we choose another approach based on the notion of a small subgroup introduced in the previous section and small cancellation theory in acylindrically hyperbolic groups developed in \cite{H13}.

More precisely, we prove in Lemma \ref{indstep} that if $H_i$ is small for some $i$, then by choosing the element $g_i$ carefully (small cancellation theory is employed here) one can ensure that the subgroup $H_{i+1}$ defined by (\ref{Hi+1}) is still small. In addition, we show that is $B$ is any finite subset of $G$ disjoint from $H_i$, we can keep it disjoint from $H_{i+1}$. Recall also that the action of $G$ on the space of left cosets $G/H$ is faithful whenever $H$ is a small subgroup and $K(G)=\{1\}$ (Lemma \ref{faith}). These results allow us to iterate the above construction in such a way that the resulting action of $G$ on $G/H_\infty$ is faithful.

\paragraph{3.2. The inductive step.}
The main goal of this subsection is to prove the following proposition, which takes care of the inductive step in the proof of Theorem \ref{main}.
Given a group $G$, a subgroup $H$, and two collections of elements $\bar a=(a_1,..., a_k)$ and $\bar (b_1,..., b_k)$ of $G$, we say that the triple $(\bar a, \bar b, H)$ is \emph{admissible} if
\begin{equation}\label{abij}
a_iH\neq a_jH\;\;\; {\rm and}\;\;\; b_iH\neq b_jH
\end{equation}
for all $1\leq i < j\leq k$.

\begin{prop}\label{indstep}
Let $G$ be a group, $H$ a small subgroup of $G$, $\bar a=(a_1, \ldots, a_k)$ and $\bar b=(b_1, \ldots , b_k)$ elements of $G^k$. Assume that the triple $(\bar a, \bar b, H)$ is admissible. Let also $B\subseteq G$ be a finite subset disjoint from $H$. Then there exists $g\in G$ and a subgroup $K\le G$ such that
\begin{enumerate}
\item[(a)] $H\le K$.
\item[(b)] $K$ is small in $G$.
\item[(c)] $ga_iK=b_iK$ for all $1\le i\le k$.
\item[(d)] $K\cap B=\emptyset$.
\end{enumerate}
\end{prop}

We first recall some well-known results and terminology concerning free products.

\begin{thm}\cite[Chapt IV, Theorem 1.2]{LS}\label{nform}
Let $G_1$, $G_2$ be arbitrary groups. Each element of a free product $G_1\ast G_2$ can be uniquely expressed as $g_1g_2...g_n$ where $n\geq 0$, each $g_i$ is a non-trivial element of one of the factors and successive elements $g_ig_{i+1}$ belong to different factors.
\end{thm}

The expression $g_1g_2...g_n$ as in the above theorem representing an element $g\in G_1\ast G_2$ is called the \emph{normal form} of the element $g$. The following corollary is an immediate consequence of the uniqueness of normal forms.

\begin{cor}\label{flength}
Let $g\in G_1\ast G_2$ and let $S_i$ be a generating set for $G_i$. Let $g_1g_2...g_n$ be the normal form for $g$, where $g_i\in G_{j_i}$. Then
\[
|g|_{S_1\cup S_2}=\sum_{i=1}^n|g_i|_{S_{j_i}}.
\]
In particular, the embeddings of metric spaces $(G_i, d_{S_i})\hookrightarrow (G_1\ast G_2, d_{S_1\cup S_2})$ are isometric.
\end{cor}

\begin{lem}\label{undisfp}
Let $G$ be a group with a generating set $A$, $H$ a subgroup of $G$ such that $H\subseteq A$. Let $\bar a=(a_1,..., a_k)$, $\bar b=(b_1,..., b_k)$ be elements of $G^k$. Assume that the triple $(\bar a, \bar b, H)$ is admissible and let $P=G\ast \langle t\rangle$,
$$
L=\langle H, b_1^{-1}ta_1,..., b_k^{-1}ta_k\rangle\le P ,
$$
and
$$
Z= H\cup \{ b_1^{-1}ta_1,..., b_k^{-1}ta_k\}\cup\{a_1^{-1}t^{-1}b_1,..., a_k^{-1}t^{-1}b_k\}.
$$
Then $L\cap G=H$ and the natural map $(L, \d_Z)\to (P, \d_{A\cup\{ t\}})$ is Lipschitz.
\end{lem}

\begin{proof}

First, if $h\in H\setminus\{1\}$, then $|h|_Z=|h|_{A\cup \{t\}}=1$.   Assume now that $h\in L\setminus H$. Let $$r_i=b_i^{-1}ta_i$$ for $1\leq i\leq k$  and let $W$ be a minimal word in $Z$ representing $h$. Then $$W= h_1r_{i_1}^{\e_1}...h_mr_{i_m}^{\e_m}h_{m+1}$$ for some $m\geq 1$, where each $h_j$ is a (possibly trivial) element of $H$ and $\e_j\in\{\pm1 \}$. Then $|h|_Z=\|W\|=\sum_{i=1}^{m+1}|h_i|_Z+m$, and note that each $|h_i|_Z\in\{0, 1\}$. Let $g_j=u_jh_jv_j$, where
\[
u_j=\begin{cases} a_{i_{j-1}} & \text{if $\e_{j-1}=1$}\\
b_{i_{j-1}} & \text{if $\e_{j-1}=-1$}
\end{cases}
\]
for $2\le j\le m+1$ and
\[
v_j=\begin{cases} b_{i_j}^{-1} & \text{if $\e_j=1$}\\
a_{i_j}^{-1} & \text{if $\e_j=-1$}
\end{cases}
\]
for $1\le j\le m$. We also define $u_1=v_{m+1}=1$. Then
\begin{equation}\label{normform}
h=g_1t^{\e_1}...g_mt^{\e_m}g_{m+1}.
\end{equation}
We show now that consecutive $t$-letters in this expression do not cancel, which means that (\ref{normform}) becomes a normal form for $h$ in the free product $G\ast \langle t\rangle$ after possibly removing the trivial $g_j$ and combining adjacent powers of $t$. Suppose for some $2\leq j\leq m$, $\e_{j-1}=1$ and $\e_{j}=-1$. Then $g_j=a_{i_{j-1}}h_ja_{i_j}^{-1}$. If $i_{j-1}\neq i_j$, then $g_j\neq 1$ since $a_{i_{j-1}}H\neq a_{i_j}H$. If $i_{j-1}=i_j$, then $h_j\neq 1$, for otherwise $r_{j-1}$ and $r_j$ would freely cancel in the word $W$. Hence $g_j\neq 1$ since it is a conjugate of $h_j$ in this case. By a similar argument, it follows that if $\e_{j-1}=-1$ and $\e_j=1$, then $g_j\neq 1$.  Thus, if $g_j=1$, then $\e_{j-1}=\e_j$ and hence the corresponding  $t$-letters cannot cancel with each other.

Since the normal form for $h$ contains at least one $t$-letter, Theorem \ref{nform} gives that $h\notin G$ and hence $L\cap G=H$. Furthermore, by Corollary \ref{flength},
\begin{equation}\label{lengthsum}
|h|_{A\cup\{t\}}=\sum_{j=1}^{m+1}|g_j|_{A}+m\geq m.
\end{equation}


Since each $|h_j|_Z\in\{0, 1\}$, it follows that
\[
|h|_{Z}=\sum_{j=1}^{m+1}|h_j|_{Z}+m\leq 3m\leq 3|h|_{\mathcal A\cup\{t\}}.
\]
\end{proof}

By definition, a group is acylindrically hyperbolic if it admits a non-elementary acylindrical action on some hyperbolic metric space. In \cite{Osi13} it is shown that an acylindrically hyperbolical group G always has such an action on a hyperbolic space of the form $\Gamma(G, A)$ where $A$ is a generating set of $G$.

\begin{defn}
Let $A$ be a generating set of $G$ such that $\Gamma(G, A)$ is hyperbolic and acylindrical. A subgroup $S\leq G$ is called \emph{suitable with respect to A} if the action of $S$ on $\Gamma(G, A)$ is non-elementary and $S$ does not normalize any non-trivial finite subgroup of $G$.
\end{defn}

The following is a modification of \cite[Theorem 7.1]{H13}; note that condition (a) is a simplification of \cite[Theorem 7.1 (d)]{H13} and condition (b) is an immediate consequence of the proof of \cite[Theorem 7.1]{H13}. Recall that $B_A(N)$ denotes the ball of radius $N$ in $G$ with respect to the metric $d_A$.

\begin{thm}\label{SCthm}
Let $A$ be a generating set of $G$ such that $\Gamma(G, A)$ is hyperbolic and acylindrical, and let $S\leq G$ be suitable with respect to $A$. Let $N\in\mathbb N$ and $t\in G$. Then there exists a group $\overline{G}$ and a surjective homomorphism $\gamma \colon G\to\overline{G}$ such that
\begin{enumerate}
\item[(a)] There exists a generating set $\overline{A}$ of $\overline{G}$ such that $\gamma(A)\subseteq \overline{A}$ and $\Gamma(\overline{G}, \overline{A})$ is hyperbolic, acylindrical, and non-elementary.

\item[(b)] $\gamma|_{B_A(N)}$ is injective.

\item[(c)] $\overline{G}=G/\ll t^{-1}w\rr$ for some $w\in S$. In particular, $\gamma(t)\in \gamma(S)$.
\end{enumerate}
\end{thm}

We are now ready to prove the main result of this subsection.

\begin{proof}[Proof of Proposition \ref{indstep}]
Since $H$ is small, there exists a generating set $X\subseteq G$ such that $H\subseteq X$ and $\Gamma (G,X)$ is hyperbolic, acylindrical, and non-elementary. By Theorem \ref{tri}, there exists $g\in G$ loxodromic with respect to the action on $\Gamma (G,X)$. Then $E(g)\h (G,X)$ (\cite[Corollary 3.16]{H13}) and hence $E(g)\h (P, X\cup \{ t\})$ by Lemma \ref{heprod} and Lemma \ref{hehe}.

Let $A=X\cup E(g)$ and let $P$, $L$, and $Z$ be as in Lemma \ref{undisfp}. Obviously $Z$ is bounded with respect to the metric $\d_{X\cup \{ t\}}$. By Lemmas \ref{undisfp} and \ref{small} applied to the group $P$, hyperbolically embedded subgroup $F=E(g)$, generating set $X\cup \{t\}$, and subgroups $S=G$ and $H=L$, we can chose a generating set $Y$ of $P$ such that $L\subseteq Y$,  $\Gamma (P, Y)$ is acylindrical, hyperbolic, and the action of $G$ on $\Gamma (P, Y)$ is non-elementary.

Since $K(G)=\{1\}$, $G$ does not normalize any non-trivial finite normal subgroup of $P$ and hence $G$ is suitable with respect to $Y$. Since $B$ is finite, we can choose $N\in\mathbb N$ such that $B\subseteq B_Y(N)$. Applying Theorem \ref{SCthm} gives a group $\overline{P}$ and a surjective homomorphism $\gamma\colon P\to \overline{P}$. By condition (c) of Theorem \ref{SCthm} and the obvious Tietze transformation, $\overline{P}\cong G$. In particular, the composition of the natural embedding $G\hookrightarrow  P$ and $\gamma$ is the identity map on $G$. Let $g=\gamma(t)$ and $K=\gamma(L)$. Since $H\leq L$, $H=\gamma(H)\leq \gamma(L)=K$. By construction $b_i^{-1}ta_i\in L$ for $1\leq i\leq k$, so the definition of $g$ and $K$ gives that $ga_iK=b_iK$ for $1\leq i\leq k$. Since $L\subseteq Y$, $K\subseteq \gamma(Y)$ and condition (a) of Theorem \ref{SCthm} gives that $K$ is small in $G$. Since $B\subseteq G$ and $B$ is disjoint from $H$, $B$ is also disjoint from $L$ by Lemma \ref{undisfp}. Finally, $L\cup B\subseteq B_Y(N)$, so condition (c) of Theorem \ref{SCthm} gives that $B\cap K=\emptyset$.
\end{proof}

\paragraph{3.3. Proof of Theorem \ref{main}.}
Below we employ the standard notation $x^y=y^{-1}xy$ for elements $x,y$ of a group.

\begin{proof}
By Lemma \ref{finrad} it suffices to prove the theorem in the case $K(G)=\{1\}$. Let $$\bigcup_{k=1}^\infty(G^{k}\times G^{k})=\{(\bar{a}_1, \bar{b}_1),(\bar{a}_2, \bar{b}_2),...\}.$$ If $\bar{a}_i=(a_1,..., a_k)$, $\bar{b}_i=(b_1,..., b_k)$, and $g\in G$, we write $g\bar aH=\bar bH$ to mean $ga_iH=b_iH$ for all $1\le i\le k$. We also enumerate all elements of $G$:
$$
G=\{ 1, g_1, g_2, \ldots \}.
$$

We will construct a sequence of subgroups $H_0\le H_1\le \ldots $
of $G$ and a sequence of elements $u_1,u_2, \ldots $ of $G$ such that the following conditions hold for every $i\in \mathbb N$.

\begin{enumerate}
\item[(a)] $H_i$ is small in $G$.

\item[(b)] For each $1\leq j\leq i$, we have $g_j^{u_j}\notin H_i$.

\item[(c)] If $i>0$ and the triple $(\bar{a}_{i},\bar{b}_{i}, H_i)$ is admissible, then there exists $t_i\in G$ such that $t_i\bar a_iH_i=\bar b_i H_i$.
\end{enumerate}

We proceed by induction. Set $H_0=\{1\}$. Assume now that we have already constructed subgroups $H_1\le \ldots\le H_i$ satisfying (a)--(c) for some $i\ge 0$.

The action of $G$ on the coset space $G/H_i$ is faithful by (a) and Lemma \ref{faith}. Hence there exists $u_{i+1}\in G$ such that $g_{i+1}^{u_{i+1}}\notin H_i$. Let
$$
B=\{ g_1^{u_1}, \ldots, g_{i+1}^{u_{i+1}}\}.
$$
Note that $B$ is disjoint from $H_i$ by part (b) of the inductive assumption and the choice of $u_{i+1}$.

Let $\bar{a}_{i+1}=(a_1,...,a_k)$ and $\bar{b}_{i+1}=(b_1,...,b_k)$. If the triple $(\bar a_{i+1},\bar b_{i+1}, H_i)$ is not admissible, we set $H_{i+1}=H_i$. All inductive assumptions obviously hold in this case. Henceforth we assume that $(\bar a_{i+1},\bar b_{i+1}, H_i)$ is admissible. Then we can apply Proposition \ref{indstep} to the small subgroup $H=H_i$ and let $H_{i+1}=K$. Conditions (a)--(c) for $H_{i+1}$ follow immediately from the proposition.

Let $H_\infty=\bigcup_{i=1}^\infty H_i$, and let $\Omega=G/H_{\infty}$ denote the corresponding coset space. We first notice that the action of $G$ on $\Omega $ is faithful. Indeed assume that some element $g_j$ acts trivially. Then all conjugates of $g_j$ by elements of $G$ belong to $H_\infty $. In particular, $g_j^{u_j} \in H_\infty $ and hence $g_j^{u_j} \in H_i $ for all sufficiently large $i$. However this contradicts (b).

Further, for any two collections of pairwise distinct cosets $a_1H_{\infty},...,a_kH_{\infty}$ and $b_1H_\infty,...,b_kH_\infty$, there exists $i$ such that $\bar{a}_i=(a_1,...,a_k)$, $\bar{b}_i=(b_1,...,b_k)$. Since $H_i\le H_\infty$, the triple $(\bar a_i, \bar b_i, H_i)$ is admissible. Then by (c) we have $t_{i}a_jH_i =b_jH_i$ for all $1\le j\le k$, which obviously implies $t_{i}a_jH_\infty =b_jH_\infty$. Thus the action of $G$ on $\Omega$ is highly transitive.
\end{proof}

\section{Transitivity degrees of mapping class groups, $3$-manifold groups, and graph products}

\paragraph{4.1. Preliminaries on permutation groups.} We begin by reviewing some basic concepts and facts about multiply transitive groups used in this section. For more details we refer the reader to \cite{C,DixMor}.

By ``countable" we mean finite or countably infinite. Let $\Omega$ denote a countable set and let $Sym (\Omega)$ denote the symmetric group on $\Omega$. We will also use the notation $Sym(\mathbb N)$ for the symmetric group on a countable infinite set. A permutation group $G\le Sym(\Omega)$ is \emph{$k$-transitive} if the action of $G$ on $\Omega $ is $k$-transitive; $G$ is \emph{primitive} if it is transitive and does not preserve any non-trivial equivalence relation (or partition) on $\Omega$. Here an equivalence relation on $\Omega $ is called \emph{trivial}, if it is equality or if all elements of $\Omega $ are equivalent.

In the following two lemmas, we collect some well-known properties of primitive and $2$-transitive groups which will be used many times throughout this section. Since the proofs are short and fairly elementary, we provide them for convenience of the reader.

\begin{lem}\label{prim}
Let $G\le Sym(\Omega)$ be a primitive permutation group.
\begin{enumerate}
\item[(a)] Every non-trivial normal subgroup of $G$ is transitive. In particular, if $G$ is infinite then it has no non-trivial finite normal subgroups.
\item[(b)] For every $\omega\in \Omega$, $Stab_G(\omega)$ is a maximal subgroup of $G$.
\end{enumerate}
\end{lem}

\begin{proof}
Let $G$ be primitive and $N\triangleleft G$. The collection of $N$-orbits defines a $G$-invariant partition of $\Omega$, hence by primitivity this partition is trivial. This means that either $N$ acts transitively on $\Omega$ or $N=\{1\}$.
Further, let $Stab_G(\omega)\leq K\leq G$. Then the $G$-translates of the orbit $K\omega$ form a $G$-invariant partition of $\Omega$. By primitivity either $K\omega=\{\omega\}$ or $K\omega=\Omega$. In the former case we have $K=Stab_G(\omega)$; in the latter case $G=K Stab_G(\omega) =K$.
\end{proof}

Recall that a group $G$ has \emph{infinite conjugacy classes}, abbreviated \emph{ICC}, if every nontrivial conjugacy class of $G$ is infinite. Note that ICC groups cannot have non-trivial center or non-trivial finite normal subgroups.

\begin{lem}\label{2-trans}
Let $G\le Sym(\Omega)$ be a $2$-transitive permutation group.
\begin{enumerate}
\item[(a)] $G$ is primitive.
\item[(b)] If $G$ is infinite, then it is ICC.
\end{enumerate}
\end{lem}

\begin{proof}
Part (a) is obvious. To prove (b), let $g\in G\setminus \{1\}$ and let $\omega_0, \omega_1,\ldots $ be distinct elements of $\Omega$. (Note that $\Omega $ is infinite since so is $G$.) Let also $\sigma$ be an element of $\Omega$ such that $g^{-1}(\sigma)\neq \sigma$. Since $G$ is 2-transitive, for all $j\geq 1$ there exists $x_j\in G$ such that $x_j(\omega_0)=g^{-1}\sigma$ and $x_j(\omega_j)=\sigma$. Then $x_j^{-1}gx_j(\omega_0)=\omega_j$,which means that $g^{x_1}, g^{x_2},...$ are all distinct conjugates of $g$. Hence $G$ is ICC.

\end{proof}

The next lemma is also quite elementary. It will be used to bound transitivity degree of various groups from above.

\begin{lem}\label{min}
Let $P\le Sym (\Omega)$ be an infinite primitive permutation group, $T$ a transitive subgroup of $P$, and $R$ a non-trivial normal subgroup of $P$. Suppose that $[T,R]=1$. Then $R$ is an infinite minimal normal subgroup of $P$.
\end{lem}

\begin{proof}
Let $R_0$ be a non-trivial subgroup of $R$ such that $R_0\lhd P$. Consider any $r\in R$ and fix some $\omega \in \Omega$. By part (a) of Lemma \ref{prim}, $R_0$ is transitive. Hence there exists $r_0\in R_0$ such that $r_0(\omega)=r(\omega)$. That is, $s(\omega)=\omega$ for $s=r^{-1}r_0$. Since $T$ is transitive, for every $\alpha \in \Omega$ there is $t\in T$ such that $\alpha =t\omega$. Since $[T,R]=1$ we obtain $$s(\alpha)=st(\omega)=ts(\omega)=t(\omega)=\alpha.$$  Thus $s$ fixes $\Omega $ pointwise, i.e., $s=1$, which in turn implies $r=r_0\in R_0$. As this is true for every $r\in R$, we obtain $R_0=R$. It remains to note that $R$ is infinite by Lemma \ref{prim} (a).
\end{proof}

\begin{defn}
We say that a group $G$ is \emph{product-like} if $G$ contains non-trivial normal subgroups $A,B$ such that $|A\cap B|=1$.
\end{defn}

The term ``product-like" reflects the fact that such groups often occur as subgroups of direct products having non-trivial intersections with multiples.

A group $G$ is said to \emph{virtually} have some property $\mathcal P$ if a finite index subgroup of $G$ has $\mathcal P$.

\begin{cor}\label{p-like}
Let $G$ be a countably infinite virtually product-like residually finite group. Then $\td (G)=1$.
\end{cor}

\begin{proof}
Suppose that there exists a faithful, $2$-transitive action of $G$ on a set $\Omega$. By our assumption, $G$ contains a finite index subgroup $G_0$ and non-trivial subgroups $A,B\lhd G_0$ such that $|A\cap B|=1$. By \cite[Theorem 7.2D]{DixMor}, a finite index subgroup in an infinite $2$-transitive permutation group is primitive. Thus $G_0$ acts primitively on $\Omega$. By Lemma \ref{prim} (a), $B$ is transitive. As $A$ and $B$ are normal in $G_0$, we have $[A,B] \le A\cap B=\{ 1\}$. Applying Lemma \ref{min} for $P=G_0$, $R=A$, and $B=T$, we obtain that $A$ is an infinite minimal normal subgroup of $G_0$. However existence of such a subgroup obviously contradicts residual finiteness of $G_0$.
\end{proof}

Note that there are many $2$-transitive solvable groups (e.g., affine groups of fields, see the discussion of affine groups after Question \ref{q1}). However, such groups cannot be residually finite.

\begin{cor}\label{sol}
Let $G$ be an infinite residually finite virtually solvable group. Then $\td (G)=1$.
\end{cor}

\begin{proof}
Let $G_0$ be a solvable normal subgroup of $G$ of finite index and let $A$ be the maximal abelian normal subgroup of $G_0$. Since $G_0$ is solvable $A\ne \{ 1\}$. Since $A$ is characteristic in $G_0$, it is normal in $G$. Assume that $G$ acts faithfully and $2$-transitively on an infinite set $\Omega$. Then $A$ is transitive by Lemma \ref{prim} (a). Since $A$ is abelian, we can apply Lemma \ref{min} to $P=G$ and $R=T=A$. As above, we conclude that $A$ is an infinite minimal normal subgroup of $G$, which contradicts residual finiteness of $G$.
\end{proof}

\paragraph{4.2. Mapping class groups}
Our main reference for mapping class groups is \cite{FM}. We review some basic definitions and necessary facts below.

Given $g,n,p\in \mathbb N\cup \{ 0\}$, let $S=\Sigma _{g,n,p}$ be a $p$ times punctured compact, connected, orientable surface of genus $g$ with $n$ boundary components. We also write $\Sigma_{g,p}$ for $\Sigma_{g,0,p}$. By the the \emph{mapping class group} of $S$, denoted ${\rm Mod}(S)$, we mean the group of isotopy classes of orientation preserving homeomorphisms of $S$; both homeomorphisms and isotopies are required to restrict to the identity on $\partial S$ (however we do allow homeomorphisms to permute the punctures).

Given some $S=\Sigma_{g,n,p}$ with $n\ne 0$, we can glue a once punctured closed disc to each boundary component of $S$. Let $S^\prime= \Sigma_{g,p+n}$ be the resulting surface. Assume, in addition, that $S$ is not a disc nor an annulus. Then there is a \emph{capping homomorphism} $\kappa\colon {\rm Mod}(S) \to {\rm Mod} (S^\prime)$ whose kernel is generated by Dehn twists corresponding to components of $\partial S$ \cite[Sections 3.6.1, 3.6.2]{FM}. Recall that a \emph{Dehn twist} $T_a$ around a curve $a$ commutes with every $f\in {\rm Mod}(S)$ fixing $a$. In particular, this implies that $Ker(\kappa)$ is a central subgroup of ${\rm Mod}(S)$.

We say that a surface $S=\Sigma_{g,p}$ is \emph{exceptional} if $3g+p\le 4$. Thus exceptional surfaces are spheres with at most $4$ punctures and tori with at most $1$ puncture. The following is a complete list of mapping class groups of exceptional surfaces (see \cite[Chapter 2]{FM}):
$${\rm Mod}(\Sigma _{1,0})= {\rm Mod}(\Sigma _{1,1})= SL(2, \mathbb Z),$$
$${\rm Mod}(S_{0,0})={\rm Mod}(S_{0,1})=\{1\},$$
and
$${\rm Mod}(\Sigma _{0,2})=\mathbb Z_2,\;\;\; {\rm Mod}(\Sigma _{0,3})=S_3, \;\;\; {\rm Mod}(\Sigma _{0,4})= (\mathbb Z_2 \times \mathbb Z_2)\rtimes PGL(2, \mathbb Z).$$

Mapping class groups of non-exceptional surfaces are acylindrically hyperbolic since they act acylindrically on the corresponding curve complexes, which are hyperbolic \cite{MM,Bow}; mapping class groups of exceptional surfaces are either finite or non-virtually cyclic and hyperbolic. Hence we obtain the following (see Section 8 of \cite{Osi13}).

\begin{lem}\label{Mod-ah}
${\rm Mod}(\Sigma_{g,p})$ is acylindrically hyperbolic for all $(g,p)\notin \{ (0,0), (0,1), (0,2), (0,3)\}$.
\end{lem}

Thus to describe highly transitive mapping class groups we need to know which mapping class groups contain  non-trivial finite normal subgroups. The following lemma is probably known to specialists. However we were unable to find a reference, so we prove it here. The idea of the proof was suggested by Dan Margalit.

\begin{lem}\label{finsub}
If $3g+p\ge 5$ and $(g,p)\ne (2,0)$, then ${\rm Mod}(\Sigma_{g,p})$ has no non-trivial finite normal subgroups.
\end{lem}

\begin{proof}
For basic properties of Dehn twists used in this proof we refer to \cite[Chapter 3]{FM}. Recall also that under our assumptions the action of ${\rm Mod}(\Sigma _{g,p})$ on the curve complex is faithful \cite{Bir,Vir}.

Assume that an element $f\in {\rm Mod}(\Sigma _{g,p})$  belongs to a finite normal subgroup $N$ and let $c$ be an essential simple closed curve on $\Sigma _{g,p}$. By the classification of subgroups generated by two distinct Dehn twists (see \cite[p. 86]{FM}), the subgroup $H=\langle T_c^2, T_{f(c)}^2\rangle $ is isomorphic to one of the groups $\mathbb Z$, $\mathbb Z^2$, or the free group of rank $2$. Note that $H\cap N=\{1\}$ since neither of the groups from the above list has torsion. On the other hand, $[T_c^2,f]=T_c^2fT_c^{-2}f^{-1}=T_c^2T_{f(c)}^{-2}\in H\cap N$ since $N$ is normal. Therefore $[T_c^2,f]=1$, which implies that $f(c)=c$ (see \cite[Section 3.3]{FM}). Since this is true for every essential simple closed curve $c$, $f$ acts trivially on the curve complex of $\Sigma_{g,p}$ and consequently $f=1$.
\end{proof}

\begin{cor}\label{mcg0}
${\rm Mod}(\Sigma _{g,n,p})$  admits a highly transitive faithful action on a countably infinite set if and only if $n=0$, $3g+p\ge 5$, and $(g,p)\ne (2,0)$.
\end{cor}

\begin{proof}
Let $G={\rm Mod}(\Sigma _{g,n,p})$. If $n\ne 0$, then either $G$ is trivial (if $\Sigma _{g,n,p}$ is a closed disk or a closed annulus) or has a non-trivial central subgroup (the kernel of the capping homomorphism) and hence $\td (G)=1$. Similarly $\td(G)=1$ if $(g,p)=(2,0)$ since $Z(G)=\mathbb Z_2$ in this case. Further if $n=0$ and $\Sigma _{g,n,p}$ is exceptional, then $G$ is either trivial or has a non-trivial finite normal subgroup, so it cannot have a highly transitive faithful action on a countably infinite set again. In all remaining cases, $G$ is acylindrically hyperbolic and has trivial finite radical by the previous lemma, so the result follows from Theorem \ref{main}.
\end{proof}

Let us now discuss subgroups of mapping class groups. The following observation will be used several times.

\begin{rem}\label{g+n}
Let $S=\Sigma_{g,n,p}$. By gluing $\Sigma _{1,1,0}$ to every boundary component of $S$, we obtain an embedding ${\rm Mod}(S)\to {\rm Mod}(\Sigma_{g+n,p})$. Thus, if we want if we want to prove certain claim for subgroups of ${\rm Mod}(\Sigma_{g,n,p})$, it suffices to prove it for $n=0$.
\end{rem}

\begin{lem}\label{pl-or-ah}
Let $G\le {\rm Mod}(\Sigma_{g,n,p})$. Assume that $G$ is non-trivial and ICC. Then  $G$ is either virtually product-like or acylindrically hyperbolic with trivial finite radical.
\end{lem}

\begin{proof}
By Remark \ref{g+n}, we can assume that $G\le {\rm Mod}(\Sigma_{g,p})$. We consider two cases.

{\it Case 1.} If $G$ contains a pseudo-Anosov element $a$, then it is either virtually cyclic or acylindrically hyperbolic. Indeed in the non-exceptional case it follows from the fact that $a$ is contained in a virtually cyclic hyperbolically embedded subgroup of $G$ \cite[Theorem 2.19]{DGO} and Theorem \ref{ah-hes}. Since $G$ is ICC, it cannot be virtually cyclic and cannot contain non-trivial finite normal subgroups. Thus $G$ is acylindrically hyperbolic with trivial finite radical.

{\it Case 2.} Now assume that $G$ does not contain a pseudo-Anosov element. Then by Ivanov's theorem \cite{Iva92} $G$ is reducible, i.e., it fixes a multicurve. Moreover, in this case there is a multicurve $c$ and a finite index subgroup $G_0\le G$ with the following properties. Let $S_1, \ldots , S_k$ be the connected
components of $S-c$. We naturally think of $S_i$ as a punctured closed surface (cuts along components of $c$ give rise to new punctures). Then $G_0$ fixes every $S_i$ setwise and the restriction of $G_0$ to $S_i$ is either trivial or infinite irreducible. This gives a homomorphism
\begin{equation}\label{reducible}
\rho \colon G_0\to {\rm Mod}(S_1)\times \cdots \times {\rm Mod}(S_k),
\end{equation}
whose kernel is generated by Dehn twists around components of $c$  \cite[Proposition 3.20]{FM}. In particular, $Ker(\rho)$ is central in $G_0$. Since $G$ is ICC, so is $G_0$ and hence $\rho $ is injective.

Let $G_1=\rho (G_0)\cap {\rm Mod}(S_1)$ and $G_2= \rho (G_0)\cap ({\rm Mod}(S_2)\times \cdots \times {\rm Mod}(S_k))$. If both $G_1$ and $G_2$ are non-trivial, then $G$ is product-like. If $G_2=\{ 1\}$, then the composition of $\rho$ and the projection map to the first factor in (\ref{reducible}) is injective and its image is either trivial or irreducible; in the latter case we complete the proof as in the first case. If $G_1=\{1\}$, then the composition of $\rho $ with the projection to ${\rm Mod}(S_2)\times \cdots \times {\rm Mod}(S_k)$ is injective and we complete the proof by induction.
\end{proof}

We will also need the following result. It is well-known to specialists (see, e.g., the remark in the end of Section 6.4.3 of \cite{FM}), but we were unable to locate the proof in the literature. We are grateful to Dan Margalit for explaining the argument to us.

\begin{lem}\label{rf}
The group ${\rm Mod}(\Sigma_{g,n,p})$ is residually finite for all $g,n,p\in \mathbb N\cup\{ 0\}$.
\end{lem}

\begin{proof}
Since a subgroup of a residually finite group is residually finite, we can assume that $n=0$ by Remark \ref{g+n}. For $p=0$, residual finiteness of ${\rm Mod}(\Sigma_{g,p})$ is a classical result of Grossman \cite{Gross}. Now assume that $p>0$.  It is easy to see that any (residually finite)-by-finite group is residually finite. Thus it suffices to prove residual finiteness of ${\rm PMod}(\Sigma_{g,p})$, the subgroup of all elements of ${\rm Mod}(\Sigma_{g,p})$ that fix the punctures. The latter group has a faithful representation by automorphisms of $\pi_1(\Sigma_{g,p-1})$ \cite{Bir69}. Since $\pi_1(\Sigma_{g,p-1})$ is residually finite and the automorphism group of a residually finite finitely generated group is residually finite \cite{Baum}, we obtain that ${\rm PMod}(\Sigma_{g,p})$ is residually finite, and hence so is ${\rm Mod}(\Sigma_{g,p})$.
\end{proof}

The next corollary refers to conditions (C$_1$)-(C$_3$) stated in the introduction.

\begin{cor}\label{mcg1}
Let $G$ be an infinite subgroup of ${\rm Mod}(\Sigma_{g,n,p})$. Then $\td(G)\in \{ 1, \infty\}$ and conditions (C$_1$)-(C$_3$) are equivalent. Moreover, (C$_1$)-(C$_3$) are also equivalent to
\begin{enumerate}
\item[(C$_4$)] $G$ is non-trivial, ICC, and not virtually product-like.
\end{enumerate}
\end{cor}

\begin{proof}
We first note that (C$_4$) $\Rightarrow$ (C$_3$) by Lemma \ref{pl-or-ah}. Further, (C$_3$) $\Rightarrow$ (C$_2$) by Theorem \ref{main} and the implication (C$_2$) $\Rightarrow$ (C$_1$) is obvious. Assume now that $G$ does not satisfy (C$_4$). If $G$ is not ICC, then $\td(G)=1$ by part (b) of Lemma \ref{2-trans}. If $G$ is virtually product-like, then  $\td (G)=1$ by Lemma \ref{rf} and Corollary \ref{p-like}. Thus $\td(G)=1$ whenever $G$ does not satisfy (C$_4$). In particular, we obtain (C$_1$) $\Rightarrow$ (C$_4$). This completes the proof of the equivalence of (C$_1$)-(C$_3$) and also shows that $\td (G)\in \{ 1, \infty\}$.
\end{proof}

\paragraph{4.3. $3$-manifold groups.}  Let $G$ be a subgroup of the fundamental group of a compact $3$-manifold. It was proved in \cite[Theorem 5.6]{MO} that one of the following three (mutually exclusive) conditions holds:
\begin{enumerate}
\item[(I)] $G$ is acylindrically hyperbolic with trivial finite radical.
\item[(II)] $G$ contains an infinite cyclic normal subgroup $Z$ and $G/Z$ is acylindrically hyperbolic.
\item[(III)] $G$ is virtually polycyclic.
\end{enumerate}
Moreover, if $G$ is itself a fundamental group of a compact irreducible $3$-manifold $M$ (this is always the case if $G$ is finitely generated by the Scott's core theorem), then we can replace condition (II) with a more geometric one.
\begin{enumerate}
\item[(II*)] $M$ is Seifert fibered.
\end{enumerate}

\begin{cor}\label{3d0}
Let $G$ be an infinite subgroup of the fundamental group of a compact $3$-manifold. Then $\td(G)\in \{ 1, \infty\}$ and  conditions (C$_1$)--(C$_3$) are equivalent.
\end{cor}

\begin{proof}
This follows immediately from Theorem \ref{main} and the above trichotomy. Indeed if $G$ satisfies (I), then $\td (G)=\infty$ by Theorem \ref{main}. If $G$ satisfies (II), it is not ICC (note that the conjugacy class of any element $z\in Z$ is contained in $\{ z, z^{-1}\}$) and hence $\td (G)=1$ by part (b) of Lemma \ref{2-trans}. Finally if $G$ satisfies (III), then $G$ cannot act primitively on a countable set by Lemma \ref{sol} and hence $\td (G)=1$. Thus conditions (C$_1$)--(C$_3$) hold simultaneously in case (I) and fail simultaneously in cases (II) and (III).
\end{proof}

Similarly, if $G$ is itself a fundamental group of a compact irreducible $3$-manifold $M$, we obtain the following.

\begin{cor}\label{3d1}
Let $M$ be a compact irreducible $3$-manifold. Then $\pi_1(M)$  admits a highly transitive faithful action on a countably infinite set iff it is not virtually solvable and $M$ is not Seifert fibered.
\end{cor}

\paragraph{4.4. Graph products.}
It is proved in \cite{MO} that for every subgroup $G$ of a finite graph product $\Gamma \{G_v\}_{v\in V}$, one of the following conditions holds.
\begin{itemize}
\item[(a)] There is a short exact sequence $\{1\}\to K\to G\to S\to \{1\}$, where $K$ is finite and $S$ is isomorphic to a subgroup of some $G_v$.

\item[(b)] $G$ is virtually cyclic;

\item[(c)] $G$ contains two infinite normal subgroups $N_1,N_2\lhd G$ such that $|N_1\cap N_2|<\infty $.

\item[(d)] $G$ is acylindrically hyperbolic.
\end{itemize}

We will also need a result of Cameron.

\begin{lem}[{\cite[Corollary 2.2]{C}}]\label{Cameron}
Suppose that $G\le Sym(\mathbb N)$ is infinite and $k$-transitive for some $k\ge 2$. Let $N$ be a non-trivial normal subgroup of $G$. Then either $N$ is $(k-1)$-transitive or $k=3$ and $N$ is an elementary abelian $2$-group (i.e., a direct sum of copies of $\mathbb Z/2\mathbb Z$).
\end{lem}

Similar to the previous cases, we obtain the following.

\begin{cor}\label{GP}
Let $G\le \Gamma \{G_v\}_{v\in V}$ be a countably infinite subgroup of a graph product of a finite family of groups $G_v$, $v\in V$. Suppose that $G$ is not isomorphic to a subgroup of some $G_v$. Then $\td (G)\in \{ 1,2,\infty\}$  and conditions (C$_1$)--(C$_3$) are equivalent. If, in addition, every $G_v$ is residually finite, then  $\td(G)\in \{ 1, \infty\}$.
\end{cor}

\begin{proof}
If $G$ satisfies (a), then $K$ must be non-trivial and hence $\td (G)=1$ by part (b) of Lemma \ref{2-trans}. Similarly $\td (G)=1$ in the case (b) as non-trivial virtually cyclic groups are not ICC. Suppose (c) holds. If $N_1\cap N_2\ne \{1\}$, then $G$ contains a non-trivial finite normal subgroup and we obtain $\td (G)=\{ 1\}$ as above. Thus it remains to consider the case $N_1\cap N_2= \{ 1\}$. We are going to prove that $\td (G)\le 2$ in this case.

Arguing by contradiction, suppose that $\td (G)\ge 3$. Let $N$ denote the (normal) subgroup of $G$ generated by $N_1$ and $N_2$. Clearly $N \cong N_1\times N_2$. Let $G$ act faithfully and $3$-transitively on a set $\Omega$. According to Lemma \ref{Cameron}, there are two possibilities.

{\it Case 1.} Neither of the subgroups $N_1$, $N_2$ is $2$-transitive. Then $N$ is an elementary abelian $2$-group, so we can apply Lemma \ref{min} to $T=R=N$ and conclude that $N$ is a minimal normal subgroup of $G$. However $N_i$, $i=1,2$ are proper subgroups of $N$. Contradiction.

{\it Case 2.} Assume now that one of the subgroups $N_1$, $N_2$, say $N_2$, is $2$-transitive.  Let $\alpha, \beta, \gamma\in \Omega$ be distinct elements. $N_1$ is transitive by Lemma \ref{prim}. Hence there exists $a\in N_1$ such that $a(\alpha)=\beta$. Since $N_2$ is $2$-transitive, there exists $b\in N_2$ such that $b(\alpha)=\alpha$ and $b(\beta)=\gamma$. Then $ab(\alpha)=\beta$ while $ba(\alpha)=\gamma$. This contradicts $[N_1, N_2]\le N_1\cap N_2=\{ 1\}$.

Notice that by a theorem of Green \cite{Gre} (see also \cite{HsuW}), a finite graph product of residually finite groups is residually finite. Thus if all multiples are residually finite, so is $G$ and instead of considering cases 1,2 above we can immediately conclude that $\td (G)=1$ by Corollary \ref{p-like}.

Finally, if $G$ satisfies (d), the conclusion follows from Theorem \ref{main} and part (b) of Lemma \ref{2-trans}.
\end{proof}

Note that the case $\td(G)=2$ in the above corollary does realize. The following example is well-known.

\begin{ex}\label{2cc}
Let $C$ be a countably infinite group where all non-trivial elements are conjugate. First examples of countable groups of this kind were constructed by G. Higman, B. Neumann and H.Neumann (see \cite{LS}); finitely generated examples can be found in \cite{Osi10}. Let $G=C\times C$.
Then $G$ naturally acts on $C$ by the rule
$$
(a,b) c = acb^{-1} .
$$
Note that this action is isomorphic to the action of $G$ on cosets $G/D$, where $D=\{ (c,c) \mid c\in C\}$.
It is easy to see that this action is $2$-transitive. Indeed let $(x_1,y_1), (x_2, y_2)$ be two pairs of distinct elements of $C$. Since all nontrivial elements of $C$ are conjugate, there exists $b\in C$ such that $x_2^{-1}y_2=b^{-1}x_1^{-1}y_1b$. Let also $a= x_2b^{-1}x_1^{-1}$. Obviously we have $ax_1b=x_2$ and
$$
ay_1b=x_2b^{-1}x_1^{-1}y_1b= x_2x_2^{-1}y_2 =y_2.
$$
Finally we observe that if $acb^{-1}=c$ for every $c\in C$, then substituting $c=1$ we obtain $a=b$. Hence $c^{-1}ac=a$ for all $c\in C$, i.e., $\{a\}$ is a conjugacy class. Clearly this implies $a=1$ and consequently $b=1$. Thus the action is faithful and we have $\td(G)\ge 2$. On the other hand, if $\td(G)\ge 3$, then both multiples of $G=C\times C$ act $2$-transitively by Lemma \ref{Cameron} and we get a contradiction arguing as as in Case 2 of the proof of Corollary \ref{GP}. Thus $\td (G)=2$.
\end{ex}

If $G$ is a finite graph product of copies of $\mathbb Z$, i.e., a right angled Artin group, then it is either $\mathbb Z$, of decomposes as a direct product of two non-trivial groups, or is acylindrically hyperbolic  \cite{MO}. Since $G$ is torsion free, its finite radical is trivial and we obtain the following.

\begin{cor}\label{RAAG}
A right angled Artin group is highly transitive if and only if it is non-cyclic and directly indecomposable.
\end{cor}

\paragraph{4.5. Miscellaneous results.}
Here we discuss applications of our main theorem to some other classes of acylindrically hyperbolic groups.

\begin{lem}
Let $G$ be a group hyperbolic relative to a collection of proper subgroups. Then $G$ is either acylindrically hyperbolic or virtually cyclic.
\end{lem}

\begin{proof}
Let $G$ be hyperbolic with respect to a collection of proper subgroups $\Hl$. We accept the definition of a relatively hyperbolic group from \cite{Osi06}. In particular, we do not assume that $G$ is countable. Withoul loss of generality we can assume that all $H_\lambda$ are non-trivial.

We consider two cases.

{\it Case 1.} There exists an infinite $H\in \Hl$. By \cite[Proposition 4.28]{DGO} we have $H\h G$. Since $H$ is infinite and $H\ne G$, it is non-degenerate and hence $G$ is acylindrically hyperbolic by Theorem \ref{ah-hes}.

{\it Case 2.} Suppose now that all $H_\lambda $ are finite. Again there are two possibilities. First assume that $\Lambda$ is finite. Then by \cite[Theorem 2.40]{Osi06} $G$ is an (ordinary) hyperbolic group. Hence $G$ is either virtually cyclic or acylindrically hyperbolic. Finally assume that $\Lambda $ is infinite. Then by part 1) of \cite[Theorem 2.44]{Osi06} $G$ decomposes as a non-trivial free product of infinitely many non-trivial groups. It is easy to see that the action on the associated Bass-Serre tree is acylindrical and non-elementary. Thus $G$ is acylindrically hyperbolic in this case.
\end{proof}

From this lemma, Theorem \ref{main} and Lemma \ref{2-trans}(b) we obtain the following.

\begin{cor}
Let $G$ be a countable group hyperbolic relative to a collection of proper subgroups. Then $G$ is highly transitive if and only if it is not virtually cyclic and has no non-trivial finite normal subgroups.
\end{cor}

\begin{cor}\label{outfn}
$Out(F_n)$ is highly transitive if and only if $n\ge 3$.
\end{cor}

\begin{proof}
Recall that $Out(F_n)$ is acylindrically hyperbolic for $n\ge 2$ \cite{DGO,Osi13}. $Out(F_2)=GL_2(\mathbb Z)$ has non-trivial finite radical, hence $\td (Out(F_2))=1$ by Lemma \ref{2-trans} (b). On the other hand, for $n\ge 3$, $Out(F_n)$ has trivial finite (and even amenable) radical. This follows immediately, for example, from $C^\ast$-simplicity of $Out (F_n)$ for $n\ge 3$ proved in \cite{BriHar}. This fact can also be derived from results of \cite{WZ}.
\end{proof}

\begin{cor}\label{1-rel}
Every $1$-relator group with at least $3$ generators is highly transitive.
\end{cor}

\begin{proof}
Let $G$ be a $1$-relator group $G$ with at least $3$ generators. It was proved in \cite{MO} that $G$ is acylindrically hyperbolic. Thus we only need to show that $G$ has trivial finite radical. It is shown in \cite{SS} (see the proof of Theorem III there) that every $1$-relator group with at least $3$ generators splits
as an HNN-extension $G = H_{A^t=B}$, where $A,B$ are proper subgroups of $H$, and there exists $h \in H$ such that
$A \cap B^h = \{1\}$. Let $T$ be the Bass-Serre tree associated to this decomposition and let $N\lhd G$ be a finite normal subgroup. Recall that every finite group acting on a tree without inversions fixes a vertex. Hence $N$ fixes a vertex of $T$. But since $N$ is normal and the action of $G$ is transitive, $N$ fixes all vertices of $T$. In particular, $N\le A \cap B^h = \{1\}$.
\end{proof}

\section{Highly transitive actions and mixed identities}

\paragraph{5.1. Mixed identity free groups.} We begin by introducing notation, which will be used throughout the rest of the paper. Let $F_n$ denote the free group of rank $n$ with basis $x_1, \ldots, x_n$. Given a group $G$, we identify elements of the free product $G\ast F_n$ with their normal forms (see Theorem \ref{nform}). Thus we think of elements of $G\ast F_n$ as words in the alphabet $\{x_1^{\pm 1}, \ldots, x_n^{\pm 1}\} \cup G$. Given $w\in G\ast F_n$ and any elements $a_1, \ldots , a_n\in G\ast F_n$, we denote by $w(a_1, \ldots, a_n)$ the word obtained from $w$ by substituting $a_i$ for $x_i$ for all $i=1, \ldots, n$. Alternatively, we can describe $w(a_1, \ldots, a_n)$ as the image of $w$ under the homomorphism $G\ast F_n \to G\ast F_n$ that is identical on $G$ and sends $x_i$ to $a_i$, $i=1, \ldots, n$.

Recall that a group $G$ is said to satisfy a \emph{mixed identity} $w=1$ for some $w\in G\ast F_n$ if $w(g_1, \ldots, g_n)=1$ for all $g_1, \ldots, g_n\in G$.  We say that a mixed identity $w=1$ is \emph{non-trivial} if $w\ne 1$ as an element of $G\ast F_n$.

\begin{rem}\label{1-vs-n}
It is easy to see that if $G\ne \{ 1\}$, then for every $n\in \mathbb N$ there is an embedding $G\ast F_n\to G\ast \langle x\rangle$ which is identical on $G$. (E.g., using normal forms in free products it is easy to show that the subgroup of $G\ast \langle x\rangle$ generated by $G$, $xgx$, \ldots, $x^ngx^n$ is naturally isomorphic to $G\ast F_n$ for every $g\in G\setminus \{ 1\}$.) Thus $G$ satisfies some non-trivial mixed identity $w\in G\ast F_n$ if and only if it satisfies a non-trivial mixed identity $w^\prime \in G\ast \langle x\rangle$ (i.e., with one variable).
\end{rem}

\begin{defn}
A group $G$ is called \emph{mixed identity free} (or MIF, for brevity) if it does not satisfy any non-trivial mixed identity.
\end{defn}

Recall that, for a given group $G$, a \emph{$G$-group} is a group $H$ with a fixed embedding $G\le H$. Let $\mathcal L_G$ denote the first order language of $G$-groups. Thus the alphabet of $\mathcal L_G$ consists of the standard logical symbols, the group operations (multiplication and taking inverse), and constants corresponding to elements of $G$. Let $Th_\forall^G(H)$ denote the universal theory of $H$ as a $G$-group. That is, $Th_\forall^G(H)$ is the set of all universal sentences in $\mathcal L_G$ that hold true in $H$.  Two $G$-groups $H_1$, $H_2$ are \emph{universally equivalent} if $Th_\forall^G(H_1)=Th_\forall^G(H_2)$. In general, being universally equivalent as $G$-groups is stronger that being universally equivalent in the usual sense (i.e., as $\{ 1\}$-groups). For details on $G$-groups and their elementary theory we refer to \cite{BMR}.

\begin{prop}\label{mif}
A countable group $G$ is MIF if and only if $G$ and $G\ast F_n$ are universally equivalent as $G$-groups for all $n\in \mathbb N$.
\end{prop}

\begin{proof}
The `if' part is obvious since satisfying a non-trivial mixed identity can be encoded by a universal sentence in $\mathcal L_G$. Thus we only need to prove the `only if' part. Assume that $G$ is MIF. Fix any $n\in \mathbb N$. Let $G\ast F_n=\{ 1, w_1,w_2, \ldots\}$. Let $G\ast F_n\to G\ast F_{n+1}=G\ast F_{n}\ast \langle x\rangle $ be the natural embedding. Let $u_1=w_1$ and for $i\ge 2, \ldots$, let $$u_i=[w_1, x^{-1}w_2x, \ldots, x^{-(i-1)}w_ix^{i-1}],$$ where the commutator is defined inductively by $$[a_1,a_2]=a_1^{-1}a_2^{-1}a_1a_2\;\;\; {\rm and} \;\;\; [a_1, \ldots , a_k, a_{k+1}]=[[a_1, \ldots, a_k], a_{k+1}].$$ Using normal forms, it is easy to see that $u_i$ is a non-trivial element of $G\ast F_{n+1}$. Since $G$ is MIF, for every $i$, there exist $g_{i1}, \ldots, g_{i,n+1}\in G$ such that $u_i(g_{i1}, \ldots, g_{i,n+1})\ne 1$. In particular, this implies that $w_m(g_{i1}, \ldots, g_{in})\ne 1$ for all $m\le i$.

Let $\omega$ be a non-principal ultrafilter on $\mathbb N$ and let $G^\omega$ denote the corresponding ultrapower of $G$. Let $g_j$ denote the element of $G^\omega$ represented by the sequence $(g_{ij})_{i\in \mathbb N}$, $j=1, \ldots, n$. We think of $G^\omega$ as a $G$-group, where the fixed embedding of $G$ is the diagonal one. Let $\e\colon G\ast F_n\to G^\omega $ be the homomorphism whose restriction on $G$ is the diagonal embedding and such that $x_j$ is mapped to $g_j$ for $j=1, \ldots, n$. By the choice of elements $g_{ij}$, we obtain that for every $m$, $w_m(g_{i1}, \ldots, g_{in})\ne 1$ for all but finitely many $i$. Since $\omega$ is non-principal, this implies $\e (w_m)\ne 1$ for every $m$. Hence $\e$ is injective. Thus we obtain a sequence of embeddings of $G$-groups: $G\to G\ast F_n\to G^\omega$. By the \L os theorem, $G$ and $G^\omega$ have the same elementary (in particular, universal) theory as $G$-groups. Since universal properties of algebraic structures are inherited by substructures, we obtain $$Th_\forall ^G (G)= Th_\forall ^G(G^\omega)\subseteq Th_\forall ^G(G\ast F_n)\subseteq Th_\forall ^G(G).$$ Consequently, $Th_\forall ^G(G\ast F_n)=Th_\forall ^G(G)$.
\end{proof}

The next proposition shows that being MIF is a much stronger property than not satisfying any nontrivial identity. Recall that the \emph{girth} of a finitely generated group $G$ is the supremum of all $n\in \mathbb N$ such that for every finite generating set $X$ of $G$, the Cayley graph $\Gamma (G,X)$ contains a cycle of length at most $n$ without self-intersections \cite{Sc}.

\begin{prop}\label{mif-prop}
For any MIF group, the following conditions hold.
\begin{enumerate}
\item[(a)] $G$ is ICC. In particular, $G$ has trivial center and contains no non-trivial finite normal subgroups.
\item[(b)] $G$ is not decomposable as direct products of two non-trivial groups.
\item[(c)] Every non-trivial subnormal subgroup of $G$ is MIF.
\item[(d)] If $G$ is finitely generated, then $G$ has infinite girth.
\end{enumerate}
\end{prop}

\begin{proof}
(a) Let $a\ne 1$ be an element whose conjugacy class in $G$ consists of $n$ elements. Then $G$ satisfies the mixed identity $[x^{n!}, a]=1$.

(b) Let $G=A\times B$. Fix any elements $a\in A$, $b\in B$. Then $G$ satisfies the mixed identity $[[a,x], [b,x]]=1$, which is obviously non-trivial provided $a\ne 1$ and $b\ne 1$.

(c) Clearly it suffices to prove that every nontrivial normal subgroup $N\lhd G$ is MIF. Let $a\in N\setminus \{ 1\}$. Assume that $N$ satisfies an identity $w(x)=1$ for some $w(x)\in N\ast \langle x\rangle$. Since $N$ is normal, $G$ satisfies the identity $w([a,x])=1$. Since $G$ is MIF, $w([a,x])=1$ as an element of $G\ast \langle x\rangle$. By using normal forms, it is easy to see that $N\ast \langle x\rangle$ is isomorphic to the subgroup of $N\ast \langle x\rangle$ generated by $N$ and $[a,x]$  via the map that takes $x$ to $[a,x]$ and is identical on $N$. This map takes $w([a,x])$ to $w(x)$ and hence $w(x)=1$ as an element of $N\ast \langle x\rangle$. Thus $N$ can only satisfy the trivial identity.

(d) Suppose that $G$ has finite girth $k$. Then for every finite generating set $g_1, \ldots, g_n$ of $G$, there exists a non-trivial element $w\in F_n$ of length at most $k$ such that $w(g_1, \ldots, g_n)=1$ in $G$. Consider any generating set $a_1, \ldots, a_n$ of $G\ast \langle x\rangle$ and let $a_i(g)$ denote the image of $a_i$ under the epimorphism of $G\ast \langle x\rangle$ that is identical on $G$ and sends $x$ to an element $g\in G\ast \langle x\rangle$. Since $a_1(g), \ldots, a_n(g)$ generate $G$ for every $g\in G$, the group $G$ satisfies the universal sentence
$$
\forall\, g\; w_1(a_1(g), \ldots , a_n(g))=1\, \vee \, \ldots\, \vee w_s(a_1(g), \ldots , a_n(g))=1,
$$
where $w_1, \ldots, w_s$ is the list of all non-trivial elements of $F_n$ of length at most $k$. By Proposition \ref{mif} we obtain that the same sentence holds true in $G\ast \langle x\rangle$. In particular, it is true for $g=x$. Since $a_i(x)=a_i$, we obtain that there is a non-trivial relation of length at most $k$ between $a_1, \ldots, a_n$. This is true for any generating set $a_1, \ldots, a_n$ of $G\ast \langle x\rangle$, so we obtain that $G\ast \langle x\rangle$ has finite girth. However this is well-known to be false (for example, see \cite[Example 2.15]{BE}).
\end{proof}

\begin{rem}
The proof of Proposition \ref{mif-prop}(d) actually shows that for all $k$, $G$ has a generating set with $n$ elements and no simple cycles of length $\leq k$, where $n$ is independent of $k$. This is equivalent to $G$ \emph{preforming} the free group $F_n$ in the language of \cite{BE} and to $G$ being without \emph{almost-identities} in the language of \cite{OS}. However, it is unclear if this is actually stronger then the property of infinite girth (see \cite[Question 8.5]{BE}).
\end{rem}

\paragraph{5.2. A dichotomy for highly transitive groups}

\begin{lem}\label{non-rel}
Let $s\in Sym(\mathbb N)$ be an arbitrary permutation, $X$ a finite subset of $\mathbb N$. Then for any permutations $a_1, \ldots, a_k\in Sym(\mathbb N)$ with infinite supports and any $\alpha_1,\ldots, \alpha_k\in \mathbb N$, there exists $t\in Sym(\mathbb N)$ such that $t\vert_X\equiv s\vert_X$ and $a_k t^{\alpha_k}\cdots a_1 t^{\alpha_1} (n)\ne n$ for some $n\in \mathbb N$.
\end{lem}

\begin{proof}
The idea of the proof is to construct a sequence of pairwise distinct numbers $n_0, m_1, n_1,\ldots, m_k, n_k\in \mathbb N$ and a permutation $t\in Sym(\mathbb N)$ such that $t\vert_X\equiv s\vert_X$ and the elements $t^{\alpha _i}$ and $a_i$ act as follows:
\begin{equation}\label{tai}
n_0\stackrel{t^{\alpha_1}}{\longrightarrow}m_1 \stackrel{a_1}{\longrightarrow} n_1 \stackrel{t^{\alpha_2}}{\longrightarrow} m_2 \stackrel{a_2}{\longrightarrow} n_2\stackrel{t^{\alpha_3}}{\longrightarrow} \ldots n_{k-1}\stackrel{t^{\alpha_k}}{\longrightarrow}m_k \stackrel{a_k}{\longrightarrow} n_k
\end{equation}
Then clearly $a_k t^{\alpha_k}\cdots a_1 t^{\alpha_1} (n_0)=n_k\ne n_0$.

We will first construct the desired sequence $n_0, m_1, n_1,\ldots, m_k, n_k$ by induction. Let $n_0$ be any element of $\mathbb N\setminus (X\cup s(X))$. Further for $k\ge i\ge 1$, we let $$Y_{i}=X\cup s(X) \cup \{ n_0\} \cup \{ m_j, n_j\mid 1\le j\le i-1\}$$ and chose any
\begin{equation}\label{ni}
m_i \in supp (a_i)\setminus (Y_{i}\cup a_i^{-1}(Y_{i})).
\end{equation}
Such a choice is always possible since $Y_i$ is finite while $supp(a_i)$ is infinite by our assumption.
We let $n_i=a_i(m_i)$. Note that (\ref{ni}) guarantees that $n_i\ne m_i$ and $\{ n_i, m_i\}$ is disjoint from $Y_i$. This completes the inductive step. Clearly the resulting sequence $n_0, m_1, n_1,\ldots, m_k, n_k$ consists of pairwise distinct numbers and does not contain any elements from $X\cup s(X)$.

Further, for each $i=1, \ldots , k$, we fix arbitrary natural numbers $l_{ij}$, $1\le j\le |\alpha_i|-1$, such that all $l_{ij}$ are pairwise distinct and are not contained in $Y_k$. Again this is possible since $Y_k$ is finite. For $i=1, \ldots , k-1$, we define a cycle $$c_i=(n_{i-1}, l_{i1}, \ldots, l_{i, |\alpha_i|-1}, m_i)^{sgn(\alpha_i)}$$ (if $|\alpha_i|=1$, then $c_i=(n_{i-1}, m_i)^{sgn(\alpha_i)}$). We also denote by $s_0$ any permutation from $Sym(\mathbb N)$ with support $supp(s_0)\subseteq X\cup s(X)$ such that $s_0(x)=s(x)$ for all $x\in X$. We now define $$t=s_0c_1\cdots c_k.$$

Note that the sets $supp (c_1), \ldots , supp(c_k)$, and $X$ are pairwise disjoint by the choice of elements $m_i$, $n_i$, and $l_{ij}$. Hence we have $t(x)=s_0(x)=s(x)$ for every $x\in X$, and $t^{\alpha_i}(n_{i-1})=c_i^{\alpha_i} (n_{i-1})=m_i$ for  $i=1, \ldots , k$. Thus $t^{\alpha _i}$ and $a_i$ act as prescribed by the diagram (\ref{tai}).
\end{proof}

Recall that a subset of a topological space is \emph{comeagre} if it is the intersection of countably many sets with dense interiors. Since $Sym(\mathbb N)$ is a Polish space with respect to the topology of pointwise convergence, every comeagre subset of $Sym(\mathbb N)$ is dense (and, in particular, non-empty) by the Baire Category Theorem. It is customary  to think of comeagre subsets as subsets including ``most" elements of the ambient space.

\begin{defn}\label{gen}
We say that a certain property \emph{holds for a generic element} of a topological space $T$ if it holds for all elements of a comeagre subset of $T$.
\end{defn}

\begin{cor}\label{fp}
Let $G$ be a highly transitive subgroup of $Sym(\mathbb N)$. Then either $G$ contains $Alt(\mathbb N)$ or for a generic element $t\in Sym(\mathbb N)$, the subgroup generated by $G$ and $t$ is naturally isomorphic to $G \ast \langle t \rangle$.
\end{cor}

\begin{proof}
Recall that by the Wielandt theorem {\cite[Theorem 5.2]{Hol}} every infinite primitive subgroup of $Sym(\mathbb N)$ containing a non-trivial permutation with finite support contains $Alt(\mathbb N)$. Thus we can assume that every non-trivial element of $G$ has infinite support.

Let $w$ be a non-trivial element of the free product $G\ast \langle x\rangle$. Recall that for $g\in G$, $w(g)$ denotes the image of $w$ under the homomorphism $G\ast \langle x\rangle\to G$ that is identical on $G$ and sends $x$ to $g$. Note that the set of solutions of any equation in a topological group is closed and hence the set $A_w=\{ g\in Sym(\mathbb N) \mid w(g)\ne 1\}$ is open. By Lemma \ref{non-rel}, $A_w$ is also dense; indeed the lemma claims that we can find an element from $A_w$ in any open neighborhood of any element $s\in G$. Hence the set $$A=\bigcap\limits_{w\in G \ast \langle x\rangle\setminus \{1\}} A_w$$ is comeager. Clearly $\langle G,t\rangle \cong G\ast \langle t\rangle $ for all $t\in A$.
\end{proof}

We are now ready to prove the main result of this section.

\begin{thm}\label{dichoto}
Let $G$ be a highly transitive countable group. Then one of the following two mutually exclusive conditions holds.
\begin{enumerate}
\item[(a)] $G$ contains a normal subgroup isomorphic to $Alt(\mathbb N)$.
\item[(b)] G is MIF.
\end{enumerate}
\end{thm}

\begin{proof}
We identify $G$ with a dense subgroup of $Sym(\mathbb N)$. By Corollary \ref{fp}, either $G$ contains a normal subgroup isomorphic to $Alt(\mathbb N)$ or there exists $t\in Sym(\mathbb N)$ such that the subgroup generated by $G$ and $t$ is naturally isomorphic to $G \ast \langle t \rangle$. In the latter case $G$ cannot satisfy any non-trivial mixed identity. Indeed, by Remark \ref{1-vs-n} it suffices to show that $G$ does not satisfy any non-trivial mixed identity with one variable. Let $w\in G\ast \langle x\rangle \setminus\{1\}$ and suppose that $w(g)=1$ for all $g\in G$. Then $w(g)=1$ for all $g\in Sym(\mathbb N)$ since the set of solutions of any equation in a topological group is closed and $G$ is dense in $Sym(\mathbb N)$. However $w(t)\ne 1$, a contradiction.

It remains to prove that conditions (a) and (b) are mutually exclusive. Assume that $G$ contains $Alt(\mathbb N)$ as a normal subgroup. Let $a=(1,2,3)$, $b=(4,5,6)$, $c=(7,8,9)$, $d=(10,11,12)$. Then for every $g\in Alt(\mathbb N)$, the elements $g^{-1}ag$, $g^{-1}bg$, $g^{-1}cg$, and $g^{-1}dg$ are $3$-cycles with disjoint supports. Hence the support of at least one of them is disjoint from $\{1,2,3\}$. This means that at least one of the commutators $[a^g,a]$, $[b^g,a]$, $[c^g,a]$, $[d^g,a]$ is trivial. Thus $Alt(\mathbb N) $ satisfies the universal sentence
\begin{equation}\label{comm}
\forall \,g\; ([a^g,a]=1\vee  [b^g,a]=1 \vee [c^g,a]=1\vee [d^g,a]=1),
\end{equation}
which obviously fails in $Alt(\mathbb N) \ast \langle x\rangle$ for $g=x$. By Proposition \ref{mif}, this implies that $Alt(\mathbb N)$ is not MIF. (In this particular case we can easily prove this fact directly by taking the iterated commutator of the four commutators in (\ref{comm}).) Finally, by part (c) of Proposition \ref{mif-prop}, we conclude that $G$ is not MIF either.
\end{proof}

It is known that any acylindrically hyperbolic group contains non-cyclic free subgroups and hence does not satisfy any non-trivial identity \cite{DGO}. It is also known that a non-cyclic torsion free hyperbolic group does not satisfy any non-trivial mixed identity \cite{AR}. We can generalizes these results as follows.

\begin{cor}\label{ah-mif}
Let $G$ be an acylindrically hyperbolic group. Then $G/K(G)$ is MIF. In particular, if $K(G)=\{ 1\}$, then $G$ is MIF.
\end{cor}

\begin{proof}
We first note that the quotient group $G/K(G)$ is acylindrically hyperbolic by Lemma \ref{finrad}, so it suffices to deal with the case $K(G)=\{ 1\}$. In this case $G$ is highly transitive by Theorem \ref{main}. Further, it is proved in \cite[Corollary 1.5]{Osi13} that every infinite normal subgroup of $G$ is acylindrically hyperbolic. $Alt(\mathbb N)$ is not acylindrically hyperbolic since every acylindrically hyperbolic group contains infinite order elements by Theorem \ref{tri}. Hence $G$ cannot contain $Alt(\mathbb N)$ as a normal subgroup. To complete the proof it remains to refer to Theorem \ref{dichoto}.
\end{proof}

A natural question arising from comparing Corollary \ref{ah-mif} and Theorem \ref{dichoto} is whether every highly transitive MIF group is acylindrically hyperbolic. We will show that the answer is negative by making use of a result from \cite{FMS}.

Let $A$ be a subgroup of a group $G$. Recall that for a subset $S\subseteq G$, the core of $A$ with respect to $S$ is defined by the formula $Core_S(A)=\bigcap\limits_{s\in S} s^{-1}As$. Further, $A$ is called \emph{highly core free} in $G$ if for every finite $F\subseteq G$, every $n\in \mathbb N$, and every subsets $S_1, \ldots, S_n\subseteq G$ such that
\begin{equation}\label{core}
G=AF \cup S_1 \cup \ldots \cup S_n,
\end{equation}
there exists $1\le k\le n$ such that $Core_{S_k}(A)=\{ 1\}$. The following is Theorem 2.2 in \cite{FMS}.

\begin{lem}\label{FMS}
Let $H$ be an HNN-extension of a countable group $G$ with associated subgroups $A$ and $B$. Assume that $A$ and $B$ are highly core free in $G$. Then $H$ is highly transitive.
\end{lem}

\begin{cor}\label{BS12}
Let $$H=\langle a,b,t\mid b^{-1}ab=a^2, [a,t]=1\rangle.$$ Then $H$ is highly transitive and MIF, but not acylindrically hyperbolic.
\end{cor}

\begin{proof}
It is easy to see that the subgroup $A=\langle a\rangle $ is highly core free in $G=\langle a,b\mid a^b=a^2\rangle$. Indeed for every finite $F\subseteq G$ and every decomposition (\ref{core}), there exists at least one $k$ such that the natural projection of $S_k$ to $\langle b\rangle$ contains infinitely many positive powers of $b$. Obviously this implies $Core_{S_k}(A)=\{ 1\}$. Hence $H$ is highly transitive by Lemma \ref{FMS}. Since $G$ is torsion free so is $H$. Thus $H$ does not contain a subgroup isomorphic to $Alt(\mathbb N)$, so $H$ is MIF by Theorem \ref{dichoto}.

Observe that $A=\langle a\rangle$ is an $s$-normal subgroup of $H$; that is, $A\cap A^h$ is infinite for every $h\in H$. By \cite[Corollary 1.5]{Osi13}, any $s$-normal subgroup of an acylindrically hyperbolic group is acylindrically hyperbolic. This implies that $H$ is not acylindrically hyperbolic.
\end{proof}

Finally, we record the following explicit construction of finitely generated highly transitive subgroups satisfying condition (a) of Theorem \ref{dichoto}.

\begin{prop}\label{ht-A}
For every finitely generated infinite group $Q$, there exists a finitely generated group $G$ such that $FSym(\mathbb N)\lhd G \le Sym(\mathbb N)$ and $G/FSym(\mathbb N) \cong Q$. In particular, $G$ is highly transitive.
\end{prop}

\begin{proof}
We identify $Sym(\mathbb N)$ and $Sym (Q)$. Let $X=\{ x_1, \ldots, x_k\}$ be a generating set of $Q$ and let $\lambda\colon Q\to Sym(Q)$ be the embedding induced by the left regular action.  Let $a_i$ denote the transposition $(1, x_i)\in Sym (Q)$ and let $G=\langle \lambda(x_1), \ldots, \lambda(x_k), a_1, \ldots, a_k\rangle\le Sym(Q)$.

Let us show that $G$ contains $FSym(Q)$. Let $t_{g,h}=(g,h)$, $g,h\in Q$, be a transposition. Assume first that $g^{-1}h=x_i$ for some $i$. Then $t_{g,h}=\lambda(g)a_i\lambda(g^{-1})$. Similarly if $g^{-1}h=x_i^{-1}$, then $t_{g,h}=t_{h,g}=\lambda(h)a_i\lambda(h^{-1})$. In both cases $t_{g,h}\in G$. In the general case, there exist elements $g_0=g, g_1, \ldots, g_n=h$, where $n=|g^{-1}h|_X$ and $g_{i+1}=g_ix_{j_i}^{\pm 1}$ for some $x_{i_j}\in X$. Then $t_{g,h}=t_{g_0,g_1}\cdots t_{g_{n-1}, g_{n}}$. Thus $t_{g,h}\in G$ for every $g,h\in Q$ and therefore $FSym(Q)\le G$.

Clearly $FSym(Q)$ is normal in $G$ since it is normal in $Sym(Q)$. Note also that $\lambda(Q)\cap FSym(Q)=\{ 1\}$. Indeed every element of $\lambda(Q)$ has single (infinite) orbit since the action of $Q$ on itself is the left regular one, while every element of $FSym(Q) $ moves only finitely many elements. Thus $G=FSym(Q) \rtimes \lambda(Q)\cong FSym(Q) \rtimes Q$.
\end{proof}

\begin{cor}\label{torsion}
There exist infinite finitely generated torsion groups admitting highly transitive faithful actions.
\end{cor}

\begin{proof}
Take a finitely generated infinite torsion group $Q$ and apply the proposition.
\end{proof}

\paragraph{5.3. Mixed identities in limits of acylindrically hyperbolic groups.}
Our next goal is to show how Corollary \ref{ah-mif} can be used to construct examples of MIF groups which are very far from being acylindrically hyperbolic. We first recall the definition of the space of marked group presentations. It is a particular case of a more general construction due to Chabauty;  for groups it was first studied by Grigorchuk \cite{Gri}.

Let $F_k$ be the free group with basis $X=\{x_1, \ldots,
x_k\}$ and let $\mathcal G_k$ denote the set of all normal subgroups of
$F_k$. Given $M,N\lhd F_k$, let
$$
\d(M,N)=\left\{
\begin{array}{ll}
\max
\left\{ \left.\frac1{|w|_X}\;\right|\; w\in N\vartriangle M
\right\},
& {\rm if} \; M\ne N\\
0,& {\rm if}\; M=N.
\end{array}
\right.
$$
It is easy to see that $(\mathcal G_k, \d)$ is a compact Hausdorff
totally disconnected (ultra)metric space \cite{Gri}.

One can naturally identify $\mathcal G_k$ with the set of all
\emph{marked $k$-generated groups}, i.e., pairs
$(G, (x_1,\ldots, x_k))$, where $G$ is a group and $(x_1, \ldots , x_k)$
is a generating $k$-tuple of $G$. For this reason the space $\mathcal G_k$ with the metric defined
above is called the \emph{space of marked groups with $k$ generators}.

Let $\ah$ denote the set of all marked $k$-generated acylindrically hyperbolic groups with trivial finite radical and let $\ahc$ denote its closure in $\mathcal G_k$. In the proposition below, the word `generic' is used in the sense of Definition \ref{gen}.

\begin{prop}\label{lim-mif}
Let $\mathcal C$ be a subset of $\mathcal G_k$ for some $k\ge 2$ such that $\mathcal C\cap \ah$ is dense in $\mathcal C$. Then a generic presentation from $\mathcal C$ defines a MIF group.
\end{prop}

\begin{proof}
Let $F_{k+1}=F(x_1, \ldots, x_k)\ast \langle x\rangle=\{ 1, w_1, w_2, \ldots \}$. If $G$ is a marked $k$-generated group (i.e., a quotient group of $F_k$), we denote by $\bar w_i$ the natural image of $w_i$ in $G\ast \langle x\rangle$. Let $\mathcal D_i$ be the set of all marked $k$-generated groups $G$ such that $\bar w_i=1$ in $G\ast \langle x\rangle$ or $\bar w_i(g)\ne 1$ in $G$ for some $g\in G$. It is easy to see that if a certain marked $k$-generated group $G$ belongs to $\mathcal D_i$, then all marked $k$-generated groups from some open neighborhood of $G$ belong to $\mathcal D_i$. Thus $\mathcal D_i$ is open. Further, let $\mathcal D=\bigcap_i \mathcal D_i$. Clearly $\mathcal D$ consists of MIF groups (see Remark \ref{1-vs-n}). By Corollary \ref{ah-mif}, $\ah \subseteq \mathcal D$. Thus $\mathcal D$ is an intersection of countably many open dense subsets of $\mathcal C$, i.e., it is comeagre in $\mathcal C$.
\end{proof}

By Proposition \ref{mif}, every countable MIF group is universally equivalent to an acylindrically hyperbolic group (namely, the free product $G\ast F_n$). The following corollary of Proposition \ref{lim-mif} shows that universal equivalence cannot be replaced by elementary equivalence and even with $\forall\exists$-equivalence here.

\begin{cor}\label{2cc}
There exist a finitely generated MIF group with exactly $2$ conjugacy classes. In particular, there exists a MIF group that is not elementary equivalent (not even $\forall\exists$-equivalent) to an acylindrically hyperbolic group.
\end{cor}

\begin{proof}
Let $\mathcal {AH}_k^{tf}$ denote the set of all marked presentations of $k$-generated torsion free acylindrically hyperbolic groups. Clearly $\mathcal {AH}_k^{tf}\subseteq \ah$. It was proved in \cite[Corollary 1.13]{H13} that a generic presentation in $\overline{\mathcal {AH}}_k^{tf}$ defines a group with $2$ conjugacy classes. On the other hand, a generic presentation in $\overline{\mathcal {AH}}_k^{tf}$ defines a MIF group by Proposition \ref{lim-mif}. Since intersection of two comeagre sets is comeagre, a generic presentation in $\overline{\mathcal {AH}}_k^{tf}$ defines a group satisfying both properties. It remains to note that $\overline{\mathcal {AH}}_k^{tf}$ is compact being a closed subset of a compact space, and hence comeagre subsets of $\overline{\mathcal {AH}}_k^{tf}$ are non-empty by the Baire category theorem.

Finally we note that groups with exactly $2$ conjugacy classes can be characterized by the $\forall\exists$-formula
$$
\forall x\, \forall y\,  \exists t\,
(x=1 \,\vee\, y=1 \,\vee\, t^{-1}xt=y).
$$
Note also that acylindrically hyperbolic groups have infinitely many conjugacy classes (and, moreover, exponential conjugacy growth, see \cite{HO1}). This implies the second claim of the corollary.
\end{proof}

Our next result provides examples of MIF groups which are not highly transitive.

\begin{cor}\label{mif-tor}
There exists MIF torsion groups of transitivity degree $1$.
\end{cor}

\begin{proof}
Let  $\mathcal C_k$ be the class of all marked $k$-generated hyperbolic groups $H$ such that
\begin{enumerate}
\item[(*)] $H$ is not virtually cyclic.
\item[(**)] For any $x,y\in H$, $[x^2,y]=1$ implies $[x,y]=1$.
\end{enumerate}
Olshanskii showed that for every $k\ge 2$, the closure $\overline{\mathcal C}_k$ contains a dense subset consisting of torsion groups (this is an immediate consequence of \cite[Theorems 3 and 4]{Ols}). Let $F_k=\{ 1, w_1, w_2, \ldots \}$. For every $i\in \mathbb N$, let $\mathcal T_i $ denote the subset of $\mathcal G_k$ consisting of marked groups $G$ such that $w_i$ has finite order in $G$. Obviously $\mathcal T_i$ is open and the set $\mathcal T=\bigcap_i \mathcal T_i$ consists of torsion groups. Combining this with the Olshanskii's result and Proposition \ref{lim-mif}, we obtain that the set of MIF torsion groups is comeagre in $\overline{\mathcal C_k}$. In particular, there exists a torsion MIF marked $k$-generated group $T\in \overline{\mathcal C}_k$.

Notice that every marked group from $\overline{\mathcal C}_k$ satisfies (**). Indeed suppose that (**) fails in some marked group $G\in \overline{\mathcal C}_k$. Then it fails in any marked group in some open neighborhood of $G$ (for the same elements $x$ and $y$ as in $G$). Hence (**) fails in some group from $\mathcal C_k$, a contradiction. In particular, $T$ satisfies (**) and therefore every involution in $T$ is central. Thus if $T$ has an involution, it has $\td (T)=1$ by Lemma \ref{2-trans} (b). If $T$ does not have involutions, then every element of $T$ has odd order and hence no subgroup of $T$ can surject on $\mathbb Z_2$. On the other hand, in any $2$-transitive permutation group acting on a set $\Omega$, the quotient of the setwise stabilizer of any pair of elements of $\Omega$ modulo the pointwise stabilizer of the same elements is isomorphic  $\mathbb Z_2$. Thus, in any case, we obtain $\td (T)=1$.
\end{proof}

\section{Open questions}

In this section we discuss some open questions about transitivity degree of countable groups. We begin with two problems mentioned in the introduction.

\begin{ques}\label{q1}
Does there exist infinite countable groups of transitivity degree $k$ for every $k\in \mathbb N$?
\end{ques}

There are examples for $k=1,2,3, \infty$. We discuss the cases $k=2$ and $k=3$ below. Starting from $k=4$ the problem seems open.

It is well-known and easy to prove that for every field $F$, the natural action of the $1$-dimensional affine group $\Aff(1,F)$ on $F$ is $2$-transitive. Recall that $\Aff(1,F)$ is the group of transformations of $F^2$ generated by the linear maps $x\mapsto x+a$ and $x\mapsto bx$ for $a\in F$ and $b\in F^\ast$. Note also that if $F$ is infinite, then $\Aff(1,F)$ does not admit any $3$-transitive faithful action. Indeed, it is easy to see that $\Aff(1,F)=F\rtimes F^\ast$. If $\Aff(1,F)$ has a faithful $3$-transitive action on a set $\Omega$ and $F\ne \mathbb F_2$, then the action of $F\lhd \Aff(1,F)$ on $\Omega$ must be $2$-transitive by the Cameron's theorem (Lemma \ref{Cameron}). However this contradicts Lemma \ref{2-trans} (b). Thus $\td (\Aff(1,F))=2$. Another example of groups of transitivity degree $2$ is given in Example \ref{2cc}.

To construct a countable group $G$ of transitivity degree $3$, we slightly modify an example from \cite{Cam1} showing that part (b) of Lemma \ref{Cameron}  does realize.  Let $V$ be the direct sum of countably many copies of $\mathbb Z_2$. We think of $V$ as a vector space over $\mathbb F_2$ with basis $e_1, e_2, \ldots $. Let $A$ be the group of automorphisms of $V$ with finite support. Let $G=A\rtimes V$ be the corresponding split extension.  The action of $V$ on itself by translations and the action of $A$ by automorphisms induces a faithful (affine) action of $G$ on $V$. Let $x,y,z$ be any distinct vectors of $V$. Using the translation action of $V$, we can shift this triple to $0,u,v$, where $u=y-x$, $v=z-x$. Note that $u\ne v$ and both vectors are non-zero; since the field of scalars is $\mathbb F_2$, $u$ and $v$ are linearly independent. Hence there is an automorphism of $V$ with finite support, which maps the triple $0,u,v$ to $0, e_1, e_2$. This shows that $\td (G)\ge 3$. On the other hand, if $\td (G)=4$, then the abelian group $V$ must be $3$-transitive by Lemma \ref{Cameron}, which contradicts Lemma \ref{2-trans} (b).

\begin{ques}
Suppose that $\td (G)=\infty$. Is $G$ highly transitive?
\end{ques}

In general, it is rather difficult to show that a certain group with $\td(G)=\infty$ \emph{does not} admit a highly transitive action. To do this, one can try to find a property of groups which can hold in a $k$-transitive group for any $k$, but never holds for a highly transitive group. We wonder if satisfying a non-trivial mixed identity is such a property. In other words, we would like to know if every countable group with $\td(G)=\infty$ is MIF.

\begin{ques}
Compute transitivity degrees of some interesting particular groups. Concrete examples include $PGL(2,F)$ for a countably infinite field $F$, free Burnside groups, Thompson's group $F$, Baumslag-Solitar groups, etc.
\end{ques}

It is well-known that the action of $PGL(2,F)$ on the corresponding projective space is $3$-transitive. Thus $\td (PGL(2,F))\ge 3$. The action of $PGL(2,F)$ on the projective space is not $4$-transitive if $F$ is infinite, but we cannot exclude the possibility that there exists another faithful action which is $4$-transitive. Thus we do not know what $\td (PGL(2,F))$ is. We do not even know whether $PGL(2,F)$ is highly transitive or not. Note that $PGL(2,F)$ is not acylindrically hyperbolic. Indeed it is either finite or simple by the Jordan-Dixon theorem, while acylindrically hyperbolic groups are never simple \cite{DGO}.

It is easy to see that transitivity degree of any solvable Baumslag-Solitar group is $1$ (e.g., by Corollary \ref{sol}). The same holds for $\langle a,t\mid t^{-1}a^mt=a^{\pm m}\rangle$; indeed in this case the conjugacy class of $a^m$ is finite and we can apply Lemma \ref{2-trans} (b). The first interesting case is $BS(2,3)=\langle a,t\mid t^{-1}a^2t=a^3\rangle$. We do not even know if $BS(2,3)$ is highly transitive.

Alexander Olshanskii (private communication) noted that transitivity degree of a free Burnside group of exponent $n$ is finite and does not exceed the maximal number $k$ such that $n$ is divisible by all positive integers $p\le k$. Indeed if $G$ acts on $\mathbb N$ faithfully and $k$-transitively, then there is a subgroup $H\le G$ (namely, the setwise stabilizer of $\{ 1, \ldots, k\}$) that admits an epimorphism onto the symmetric group $S_k$. Hence for every positive integer $p\le k$, $H$ must contain an element whose order is divisible by $p$. In particular, transitivity degree of a free Burnside group of odd exponent is $1$. It is quite possible that this is also true for all infinite free Burnside groups of sufficiently large exponent.

\begin{ques}
Is being highly transitive a weakly geometric property? That is, if $G$ is quasi-isometric to a highly transitive group, is it commensurable to such a group?
\end{ques}

Note that being highly transitive passes to subgroups of finite index but not to finite extensions.  So if we replace ``commensurable" with ``isomorphic", the answer is negative.

\begin{ques}\label{amen}
Does there exist a countable amenable highly transitive subgroup $A\le Sym(\mathbb N)$ such that $A\cap Alt(\mathbb N) =\{1\}$?
\end{ques}

If we drop the assumption $A\cap Alt(\mathbb N) =\{1\}$, then there are even finitely presented examples.

\begin{ex}\label{Hough}
Fix some $n\in \mathbb N$ and let $\Omega_n$ be
the disjoint union of $n$ copies $\mathbb N_1, \ldots, \mathbb N_n$ of $\mathbb N$. Let $H_n$ be the group of all permutations $s$ of
$\Omega_n$ such that $s$ acts as a translation along each $N_i$ outside a finite set. These groups, commonly called Houghton groups, were introduced in \cite{Hough}. It is known that $H_n$ if finitely generated for $n\ge 2$ and finitely presented for $n\ge 3$. Further it is not hard to show that for $n\ge 2$, the derived subgroup of $H_n$ coincides
with $FSym (\Omega _n)$. Thus $H_n$ is (locally finite)-by-abelian (in particular, it is elementary amenable) and highly transitive. For details, see \cite{Bro}.
\end{ex}

As the first step towards answering Question \ref{amen}, we propose the following.

\begin{ques}
Construct countable (or, better, finitely generated) MIF amenable groups.
\end{ques}

It is likely that lacunary hyperbolic amenable groups constructed in \cite{OOS} are MIF. Recall that a group $G$ is \emph{lacunary hyperbolic} if it is finitely generated and at least one asymptotic cone of $G$ is an $\mathbb R$-tree.  In fact, we conjecture the following.

\begin{conj}
Every lacunary hyperbolic group that is not virtually cyclic and has no finite normal subgroups is MIF.
\end{conj}

This should follow from the following fact: Every non-trivial mixed identity $w\in G\ast \langle x\rangle$ in a $\delta$-hyperbolic group $G$ with trivial finite radical can be violated by an element $g\in G$ of length at most $C(\delta+1) (s(w)+1)$, where $s(w)$ is the sum of lengths (in $G$) of elements $a_1, \ldots, a_k\in G$ involved in the normal form of $w$. It seems plausible that this result can be proved by taking $g$ to be a high power of an infinite order element $h\in G$ such that no $a_i$ belongs to $E(h)$ (i.e., to the maximal virtually cyclic subgroup of $G$ containing $h$). However verifying details of this proof is outside of the scope of our paper.

Lacunary hyperbolic groups can be thought of as infinitely presented analogues of hyperbolic groups. We note, however, that Chainikov's theorem about highly transitive actions of hyperbolic groups does not extend to lacunary hyperbolic ones.

\begin{ex}
There exists a non-virtually cyclic torsion free lacunary hyperbolic group $G$ such that $td(G)=1$.
Indeed let $G$ be a non-cyclic torsion free lacunary hyperbolic group such that every proper subgroup of $G$ is cyclic. Such groups were constructed in \cite{OOS}. Assume that $G$ acts $2$-transitively on a set $X$ and let $H$ be the stabilizer of some $x\in X$. Then $H$ is infinite cyclic. However Mazurov \cite{Maz} proved that there are no $2$-transitive permutation groups with infinite cyclic point stabilizers.
\end{ex}


\vspace{.5cm}

\noindent {\bf M. Hull:  } MSCS UIC 322 SEO, M/C 249, 851 S. Morgan St.,
Chicago, IL 60607-7045, USA\\
E-mail: {\it mbhull@uic.edu}

\noindent {\bf D. Osin:  } Stevenson Center 1326, Department of Mathematics, Vanderbilt University, Nashville, TN 37240, USA\\
E-mail: {\it denis.osin@gmail.com}

\end{document}